\documentclass[10.5pt,a4paper]{article}
\usepackage{amsfonts,amssymb}\usepackage{bbm}
\usepackage{graphicx,latexsym,euscript,makeidx,color,bm}
\usepackage{amsmath,amsfonts,amssymb,amsthm,thmtools,mathrsfs,enumerate}
\usepackage[colorlinks,linkcolor=blue,anchorcolor=green,citecolor=red]{hyperref}
\usepackage{graphicx,tikz,exscale,pagecolor}  

\usepackage{booktabs,subfigure,xcolor}

\usepackage{geometry}
 \def\sB{\mathscr{B}}  \def\cB{{\cal B}}

\def\dbE{\mathbb{E}}     
\def\dbF{\mathbb{F}} \def\sF{\mathscr{F}}  \def\cF{{\cal F}}  
     
 \def\sH{\mathscr{H}}  \def\cH{{\cal H}}

   \def\cL{{\cal L}}  
 \def\sM{\mathscr{M}}  \def\cM{{\cal M}}  
\def\dbN{\mathbb{N}}     
     
\def\dbP{\mathbb{P}}     
     
\def\dbR{\mathbb{R}}

\def\dbY{\mathbb{Y}}   \def\cY{{\cal Y}}  \def\By{{\bm y}}
   \def\cZ{{\cal Z}}  
\def\By{{\bf y}}
\def\ss{\smallskip}             \def\hb{\hbox}
\def\ms{\medskip}              \def\ae{\hbox{\rm a.e.}}
\def\bs{\bigskip}            \def\as{\hbox{\rm a.s.}}
\def\ds{\displaystyle}       
\def\ts{\textstyle}         
\def\no{\noindent}          
\def\ns{\noalign{\ss}}     
     
\def\rf{\eqref}            
         \def\hp{\hphantom}
\def\deq{\triangleq}     \def\({\Big (}       \def\nn{\nonumber}
\def\les{\leqslant}      \def\){\Big )}       
\def\ges{\geqslant}      \def\[{\Big[}        
\def\ti{\tilde}          \def\]{\Big]}        
\def\wt{\widetilde}      \def\q{\quad}        
\def\h{\widehat}         \def\qq{\qquad}      \def\1n{\negthinspace}
\def\cd{\cdot}           \def\2n{\1n\1n}      
\def\cds{\cdots}         \def\3n{\1n\2n}

\def\h{\widehat}
\def\wt{\widetilde}
\def\ti{\tilde}
\def\cd{\cdot}
\def\cds{\cdots}

\def\deq{\triangleq}
\def\les{\leqslant}
\def\ges{\geqslant}

\def\a{\alpha}        \def\G{\Gamma}      \def\Om{\Omega}   \def\om{\omega}
       \def\D{\Delta}   \def\d{\delta}        
\def\z{\zeta}         \def\Th{\Theta}  \def\th{\theta}    
\def\e{\varepsilon}     \def\l{\lambda}        
    \def\t{\tau}     \def\f{\varphi}     

\def\ba{\begin{array}}                \def\ea{\end{array}}
\def\bel{\begin{equation}\label}      \def\ee{\end{equation}}

\newtheoremstyle{indented}{}{}{\it}{\parindent}{\bfseries}{.}{.5em}{}
\theoremstyle{indented}

\newtheorem{theorem}{Theorem}[section]
\newtheorem{definition}[theorem]{Definition}
\newtheorem{proposition}[theorem]{Proposition}
\newtheorem{corollary}[theorem]{Corollary}
\newtheorem{lemma}[theorem]{Lemma}
\newtheorem{remark}[theorem]{Remark}
\newtheorem{example}[theorem]{Example}

\makeatletter
   
   \@addtoreset{equation}{section}
\makeatother
  \allowdisplaybreaks[4]
\def\be{\begin{equation}}
\def\bel{\begin{equation}\label}
\def\ee{\end{equation}}
\def\bea{\begin{eqnarray}}
\def\eea{\end{eqnarray}}
\def\bt{\begin{theorem}\label}
\def\et{\end{theorem}}
\def\bc{\begin{corollary}\label}
\def\ec{\end{corollary}}
\def\bex{\begin{example}\label}
\def\ex{\end{example}}
\def\bl{\begin{lemma}\label}
\def\el{\end{lemma}}
\def\bp{\begin{proposition}\label}
\def\ep{\end{proposition}}
\def\br{\begin{remark}\label}
\def\er{\end{remark}}
\def\ba{\begin{array}}
\def\ea{\end{array}}
\def\bde{\begin{definition}\label}
\def\ede{\end{definition}}
\def\sqr#1#2{{\vcenter{\vbox{\hrule height.#2pt
              \hbox{\vrule width.#2pt height#1pt \kern#1pt \vrule width.#2pt}
              \hrule height.#2pt}}}}
\def\signed #1{{\unskip\nobreak\hfil\penalty50
              \hskip2em\hbox{}\nobreak\hfil#1
              \parfillskip=0pt \finalhyphendemerits=0 \par}}
\def\endpf{\signed {$\sqr69$}}

\begin{document}
\title{\bf Backward Stochastic Differential Equations and Backward Stochastic Volterra Integral Equations with Anticipating Generators}

\author{
Hanxiao Wang\thanks{College of Mathematics and Statistics, Shenzhen University, Shenzhen 518060, China
 (Email: {\tt hxwang@szu.edu.cn}). This author would like to thank  Chenchen Mou
  (of City University of Hong Kong) for some useful discussions.},~~~
Jiongmin Yong\thanks{Department of Mathematics, University of Central Florida,
                           Orlando, FL 32816, USA (Email: {\tt jiongmin.yong@ucf.edu}).
                           This author is supported in part by NSF Grant DMS-1812921.},
                          ~~~
Chao Zhou\thanks{ Department of Mathematics and Risk Management Institute, National University of Singapore,
                           Singapore 119076, Singapore (Email: {\tt matzc@nus.edu.sg}).
                           This author is supported by  NSFC Grant 11871364
                           as well as Singapore MOE  AcRF Grants R-146-000-271-112 and R-146-000-284-114.}}

\maketitle

\centerline{({\it This paper is dedicated to Professor Alain Bensoussan on the occasion of his 80th birthday})}

\bs

\no\bf Abstract. \rm For a backward stochastic differential equation (BSDE, for short), when the generator is not progressively measurable, it might not admit adapted solutions, shown by an example. However, for backward stochastic Volterra integral equations (BSVIEs, for short), the generators are allowed to be anticipating. This gives, among other things, an essential difference between BSDEs and BSVIEs. Under some proper conditions, the well-posedness of such kind of BSVIEs is established. Further,  the results are extended to path-dependent BSVIEs, in which the generators can depend on the future paths of unknown processes. An additional finding is that for path-dependent BSVIEs, in general, the situation of anticipating generators is not avoidable and
the adaptedness condition similar to that imposed for anticipated BSDEs by Peng--Yang \cite{Peng-Yang2009}
 is not necessary.

\ms

\no\bf Keywords. \rm backward stochastic Volterra integral equation, backward stochastic differential equation, anticipating generator,
path-dependence.

\ms

\no\bf AMS subject classifications. \rm 60H20, 60H10, 45D05, 45G10.

\section{Introduction}
Let $(\Om,\cF,\dbP)$ be a complete probability space on which a standard one-dimensional
Brownian motion $W=\{W(t);0\les t<\infty\}$ is defined.
The augmented natural filtration of $W$ is denoted by $\dbF=\{\cF_t\}_{t\ges0}$.
Let $T>0$ be a fixed time horizon. We denote
$$
L^p_{\cF_T}(\Om;\dbR^n)=\big\{\xi:\Om\to\dbR^n~| ~\xi \hbox{ is  $\cF_{T}$-measurable, } \dbE|\xi|^p<\infty\big\}.
$$
For any given $\xi\in L^p_{\cF_T}(\Om;\dbR^n)$ with $p>1$, let us consider the following backward stochastic differential equation (BSDE, for short) in its integral form:
\bel{BSDE}Y(t)=\xi+\int_t^T g(s,Y(s),Z(s))ds-\int_t^TZ(s)dW(s),\q t\in[0,T],\ee
where $g:[0,T]\times\dbR^n\times\dbR^n\times\Om\to\dbR^n$ is a given $\cB([0,T]\times\dbR^n\times\dbR^n)\otimes\cF_T$-measurable map, which is called the {\it generator}.
Any pair $(Y(\cd),Z(\cd))$ of $\dbF$-adapted processes satisfying \rf{BSDE} in the It\^o's sense is called an {\it adapted solution}. To guarantee the well-posedness of \rf{BSDE}, the following standard conditions are normally assumed.

\ms

{\bf(H1)$_p$} The process $s\mapsto g(s,y,z)$ is $\dbF$-progressively measurable
on $[0,T]$, for all $(y,z)\in\dbR^n\times\dbR^n$. The map $(y,z)\mapsto g(s,y,z)$
is uniformly Lipschitz continuous, and
$$\dbE\Big(\int_0^T|g(s,0,0)|ds\Big)^p<\infty. $$

In the above, the uniformly Lipschitz condition can be relaxed a little. For
convenience, if $s\mapsto g(s,y,z)$ is $\dbF$-progressively measurable, we refer
to $g(\cd)$ as an {\it adapted generator}. If $s\mapsto g(s,y,z)$ is merely
$\sF_T$-measurable, we refer to it as an {\it anticipating generator}. The following
result is by now standard (see Pardoux--Peng \cite{Pardoux-Peng1990} and El Karoui--Peng--Quenez \cite{El Karoui-Peng-Quenez1997}, for example).

\begin{theorem}\label{BSDE-well-posed} \sl Let {\rm(H1)$_p$} hold with $p>1$. Then, for any
$\xi\in L^p_{\cF_T}(\Om;\dbR^n)$, BSDE \rf{BSDE} admits a unique adapted solution
$(Y(\cd),Z(\cd))$ and the following holds:
\bel{BSDE-estimate1}\dbE\[\sup_{t\in[0,T]}|Y(t)|^p+\(\int_0^T|Z(s)|^2ds\)^{p
\over2}\]\les K\dbE\[|\xi|^p+\(\int_0^T|g(s,0,0)|ds\)^p\].\ee
Hereafter, $K>0$ stands for a generic constant which could be different from line to
line. Moreover, if $\bar g(\cd)$ is another generator satisfying {\rm(H1)$_p$} as well,
and $\bar\xi\in L^p_{\sF_T}(\Om;\dbR^n)$, then by letting $(\bar Y(\cd),\bar Z(\cd))$
be the corresponding adapted solution, one has the following stability estimate:
\begin{align}
&\dbE\[\sup_{t\in[0,T]}|Y(t)-\bar Y(t)|^p+\(\int_0^T|Z(s)-\bar Z(s)|^2ds\)^{p\over2}\]\nn\\
&\q\les K\dbE\[|\xi-\bar\xi|^p+\(\int_0^T\big|g(s,Y(s),Z(s))-\bar g(s,Y(s),Z(s))\big|ds\)^p\].\label{BSDE-stability}
\end{align}
\end{theorem}

\ms
We point out that under (H1)$_p$, the generator of BSDE \rf{BSDE} is adapted. Therefore,
it is natural to ask: \emph{Can the generator $g(\cd)$ of BSDEs be anticipating?} From
\rf{BSDE}, one intuitively sees that although the sum of the first two terms on the
right-hand side is  $\sF_T$-measurable, but the last term is an It\^o's integral
over $[t,T]$, which might make the whole sum $\dbF$-adapted. However, this is not that
simple. Let us present the following simple example.

\begin{example}\label{eample1} \rm Consider the following BSDE:
\bel{BSDE-Example}Y(t)=\int_t^TW(T)ds-\int_t^TZ(s)dW(s),\qq t\in[0,T].\ee
We see that the above BSDE has the zero terminal state and has the generator:
$$g(s,y,z)=W(T),\qq(s,y,z)\in[0,T]\times\dbR\times\dbR,$$
which is anticipating, and independent of $(y,z)$. We claim that BSDE \rf{BSDE-Example} does not have an adapted solution $(Y(\cd),Z(\cd))$. To see this, let us denote
$$\psi(t)=\int_t^TW(T)ds=W(T)(T-t),\qq t\in[0,T].$$
Clearly, $\psi(\cd)$ is an $\cF_T$-measurable process (not $\dbF$-adapted).
Then, \rf{BSDE-Example} can be rewritten as the following
\bel{BSDE-Example1}Y(t)=\psi(t)-\int_t^T Z(s)dW(s),\qq t\in[0,T].\ee
Now, we introduce the following equation:
\bel{BSVIE-Example}\cY(t)=\psi(t)-\int_t^T \cZ(t,s)dW(s),\qq t\in[0,T],\ee
which is a special form of backward stochastic Volterra integral equations (BSVIEs, for short; see the next section for details). From Yong \cite{Yong2008}, it is known  that
such a BSVIE admits a unique adapted solution $(\cY(\cd),\cZ(\cd,\cd))$, which has the explicit form:
$$
(\cY(s),\,\cZ(t,s))=(W(s)(T-s),\,T-t),\qq 0\les t\les s\les T.
$$
Thus, if \rf{BSDE-Example} (i.e., \rf{BSDE-Example1}) has an adapted solution $(Y(\cd),Z(\cd))$, then we must
have
\bel{1.8}Z(s)=\cZ(t,s)=T-t,\qq 0\les t\les s\les T,\ee
because the adapted solution of \rf{BSDE-Example} is also an adapted solution of
BSVIE \rf{BSVIE-Example}. Clearly, \rf{1.8} cannot be true.

\end{example}

The above example tells us the following: (i) BSDEs with anticipating generators might not have adapted solutions, and they could have adapted solutions within the framework of BSVIEs; (ii) BSVIEs with anticipating generators might have adapted solutions, under some additional conditions. The purpose of this paper is to explore the results of BSVIEs with anticipating generators. We will establish a general theory for such kind of BSVIEs, and beyond.

\ms

Besides the mathematical considerations, one might be wondering if there are some
practical motivations for BSVIEs with anticipating generators. We now describe some
possible situations.

\ms

Consider a position process $\psi(\cd)$ on the time horizon $[0,T]$, which is
$\sF_T$-measurable (not necessarily $\dbF$-adapted). As in Yong \cite{Yong2007} and Wang--Sun--Yong \cite{Wang-Sun-Yong2021},
one may use a BSVIE to describe the equilibrium dynamic risk measure of
$\psi(\cd)$\footnote{One may also call it the equilibrium utility/disutility of
$\psi(\cd)$.}.
More precisely, if $(Y(\cd),Z(\cd\,,\cd))$ is an adapted solution to the following
BSVIE:
$$Y(t)=\psi(t)+\int_t^Tg(t,s,Y(s),Z(t,s))ds-\int_t^TZ(t,s)dW(s),\qq t\in[0.T],$$
for some suitable generator $g(\cd)$, then
$$\rho(t,\psi(\cd))=-Y(t),\qq t\in[0,T],$$
is an {\it equilibrium dynamic risk measure} of position process $\psi(\cd)$. It is not
hard to understand that the generator $g(\cd)$ could be a deterministic function, or an
$\dbF$-progressively measurable random field. These cases were covered in the literature; see \cite{Yong2008,Wang-Sun-Yong2021}, for examples. We now describe some situations that the generator $g(\cd)$ is anticipating (or $\sF_T$-measurable). Suppose the environment (in which a BSVIE is considered) is depending on some random events, represented by a random variable, say, $\th$. These random events could include something like default payments, bankruptcy of some relevant companies/institutes, the approval/disapproval of FDA for some medicines, election results, pass/not pass some tax bills, possible in-coming monetary policies, some possible regional wars, to mention a few. The common feature of all these is that they are not $\dbF$-adapted, and even not
$\cF_T$-measurable. However, by the time $T$, it will be known whether such
an event has happened. Therefore, we may accept that the generator of the BSVIE is of form $g(t,s,\dbE[\th\,|\,\cF_T],y,z)$, or simply allow $g(\cd)$ to be
$\sF_T$-measurable\footnote{Of course, one could pursuit a more general framework.
However, at this stage, we are satisfied with such a formulation.}.

\ms

In this paper, we will establish the well-posedness of BSVIEs with anticipating
generators, under proper conditions, which includes the situation of adapted
generators. The results will be extended to the path-dependent BSVIE, which is an essential extension of
the so-called {\it anticipated BSDEs} introduced by Peng--Yang \cite{Peng-Yang2009}.
Let us mention that in studying path-dependent BSVIEs, one will see that the situation
of anticipating generators is not avoidable, unless some restrictive (or special)
cases are discussed. We will explain this in more details in a later section.

\ms

The rest of the paper is organized as follows. In \autoref{Subsec:literature}, we provide a literature review on the theory and applications of BSVIEs.
In \autoref{Sec:Pre}, we give the preliminaries and recall the approach developed by Yong \cite{Yong2008} to the well-posedness of BSVIEs with adapted generators.
In \autoref{Sec:Well-posedness}, we establish the well-posedness of BSVIEs with anticipating generators under proper conditions.
The results of path-dependent BSVIEs are presented in  \autoref{Sec:path-dependent-BSVIE}.
In \autoref{Sec:BSVIE-FBSDE}, we establish a connection between BSVIEs and coupled FBSDEs.
Some concluding remarks are presented in \autoref{Sec:Concludng} and
some technical results are collected in Appendix.

\subsection{Literature Review on BSVIEs}\label{Subsec:literature}
As a natural extension of BSDEs, BSVIE was initially studied by Lin \cite{Lin-2002} in a special form of the following:
\bel{BSVIE-I}Y(t)=\psi(t)+\int_t^Tg(t,s,Y(s),Z(t,s))ds-\int_t^TZ(t,s)dW(s),\qq
t\in[0,T].\ee
Motivated by the study of optimal control for forward SVIEs,
Yong \cite{Yong2008} introduced the notion of {\it adapted M-solution} and the following more general  form of BSVIEs:
\bel{BSVIE-II}Y(t)=\psi(t)+\int_t^Tg(t,s,Y(s),Z(t,s),Z(s,t))ds-\int_t^TZ(t,s)dW(s),\qq
t\in[0,T].\ee
The difference is that $Z(s,t)$ appears in \rf{BSVIE-II}, not in \rf{BSVIE-I}.
Naturally, we call \rf{BSVIE-I} a {\it Type-I BSVIE} and \rf{BSVIE-II} a {\it Type-II BSVIE}, respectively.
This important development has triggered extensive research on BSVIEs
and their applications. For instance, the well-posedness of BSVIEs was investigated by
Wang--Zhang \cite{Wang-Zhang2007}, Ren \cite{Ren2010}, Anh--Grecksch--Yong \cite{Anh-Grecksch-Yong2011}, Shi--Wang--Yong \cite{Shi-Wang-Yong2013}, Shi--Wen--Xiong \cite{Shi-Wen-Xiong2020}, Wang--Sun--Yong \cite{Wang-Sun-Yong2021}, and Hamaguchi \cite{Hamaguchi2021-COCV} in different settings;
Wang--Yong \cite{Wang-Yong2015} established  the comparison theorems of BSVIEs;
the variation of constants formula for linear BSVIEs was developed by Hu--{\O}ksendal \cite{Hu2018},
Wang--Yong--Zhang \cite{Wang-Yong-Zhang2021} and Hamaguchi \cite{Hamaguchi2021};
the Feynman--Kac  formula for BSVIEs was  established by Wang--Yong \cite{T.Wang-Yong2019}, Wang \cite{H.Wang2021},
and Wang--Yong--Zhang \cite{Wang-Yong-Zhang2021} within  Markovian and non-Markovian frameworks, respectively.
It is noteworthy that in \cite{Wang-Yong-Zhang2021},  the forward and backward equations are both of Volterra type,
and then the associated PDE is non-local and path-dependent, due to which the corresponding viscosity solution theory
is challenging. A fundamental tool is the functional It\^o's formula established by Viens--Zhang \cite{Viens-Zhang2019}.
Some applications of BSVIEs on optimal controls and mathematical finance
can be found in  \cite{Agram-Oksendal2015,Shi-Wang-Yong 2015,T.Wang-Zhang2017,T.Wang-2020,Wang-Yong2022}
and \cite{Yong2007,Di Persio2014,Kromer-Overbeck2017,Agram2018,Wang-Sun-Yong2021}, respectively.
In particular, we remark that using BSVIEs, Wang--Yong \cite{Wang-Yong2022} recently developed a spike
variation method for controlled forward SVIEs.

\ms
Besides the above works, another remarkable development of BSVIEs is the finding of its connection with time-inconsistent control problems. Indeed, BSVIE has become a popular tool for studying time-inconsistent problems recently.
See, for example, Wang--Sun--Yong \cite{Wang-Sun-Yong2021}, Wang--Yong \cite{H.Wang-Yong2021}, Hern\'{a}ndez--Possama\"{\i}   \cite{Hernandez2020,HP2}, Hamaguchi \cite{Hamaguchi2020},  Agram--Djehiche \cite{Agram-Djehiche2020},
Lei--Pun \cite{Lei-Pun2021} and  Hern\'{a}ndez \cite{Hernandez2021}.
From Yong \cite{Yong2012,Yong2014}, we know that one of the main reasons of causing time-consistency is the so-called
{\it subjective time-preferences}, which,  mathematically, can be described by the non-exponential discounting.
It is shown \cite{Wang-Sun-Yong2021,H.Wang-Yong2021} that the most suitable
dynamic recursive cost functional allowing non-exponential discounting should be described by a Type-I
BSVIE. As a Feynman--Kac representation of the equilibrium HJB equation (or called an extended HJB equation) associated with time-inconsistent problems,
Wang--Yong \cite{H.Wang-Yong2021}  introduced and initially investigated the following  BSVIE:
\bel{BSVIE-III}
Y(t)=\psi(t)+\int_t^Tg(t,s,Y(s),Z(s,s))ds-\int_t^TZ(t,s)dW(s),\qq
t\in[0,T].\ee
The main feature of BSVIE \rf{BSVIE-III}
is that the diagonal value $Z(s,s)$ of $Z(\cd\,,\cd)$ appears in the generator.
In the Markovian case, Wang--Yong \cite{H.Wang-Yong2021} proved that \rf{BSVIE-III} admits an adapted solution by an analytic approach, and the uniqueness of adapted solutions was also discussed.
The general well-posedness of such kind of BSVIEs was recently studied by Hern\'{a}ndez--Possama\"{\i} \cite{HP2}.
A key step before \cite{HP2} is that  Hamaguchi  \cite{Hamaguchi2020}  gave a proper definition for $Z(s,s)$,
using the derivatives of $Z(t,s)$ with respect to $t$.
In \cite{Wang-Yong-Zhou2022}, we coined the name of {\it Type-III BSVIEs} for \rf{BSVIE-III}
to distinguish it from the other two types of BSVIEs.
Moreover, in \cite{Wang-Yong-Zhou2022} we found that  Type-III BSVIEs can be used to represent the M-solutions of linear Type-II BSVIEs,
which serves as a bridge to the decoupling approach of  stochastic Hamiltonian systems of Volterra type.

\ms
We emphasize that the generators of BSVIEs in all the above works  are assumed to be adapted. In the current paper, we would like to study the results of BSVIEs with anticipating generators.
More precisely, we shall focus on the Type-I and Type-II BSVIEs,
and  the results of Type-III BSVIEs  will be reported in our future publications.

\section{Preliminaries}\label{Sec:Pre}

In this section, we will recall a standard approach in establishing the well-posedness of BSVIEs with adapted generators, from which, we will see why such a standard method does not apply to the case of anticipating generators.

\ms

Recall that in \rf{BSVIE-I} and \rf{BSVIE-II}, $g(\cd)$ and $\psi(\cd)$ are called the {\it generator} and the
{\it free term} of the BSVIE. The solution that we are looking for is a pair
$(Y(\cd),Z(\cd\,,\cd))$ of processes satisfying certain conditions. In order to make
the above more precise, let us denote
$$\D[0,T]=\big\{(t,s)\in[0,T]^2\bigm|0\les t<s\les T\big\},$$
which is an upper triangle domain in $[0,T]^2$. Note that $\D[0,T]$ does not contain the diagonal segment $\{(s,s)\bigm|s\in[0,T]\}$. This will allow our BSVIEs to have some singularities along the diagonal segment. We now introduce some spaces. Let $\cH^p_\D[0,T]$ be the set of all pairs $(Y(\cd),
Z(\cd\,,\cd))$ of processes having the following properties: $Y:[0,T]\times\Om\to
\dbR^n$ is $\dbF$-adapted, $Z:\D[0,T]\times\Om\to\dbR^n$ is
$\cB(\D[0,T])\otimes\cF_T$-measurable, for each $t\in[0,T]$, $s\mapsto Z(t,s)$ is
$\dbF$-adapted on $[t,T]$, and
$$\|(Y(\cd),Z(\cd\,,\cd))\|_{\cH^p_\D[0,T]}\equiv\Big\{\dbE\[\int_0^T|Y(t)|^pdt
+\int_0^T\(\int_t^T|Z(t,s)|^2ds\)^{p\over2}dt\]\Big\}^{1\over p}<\infty.$$
Clearly, the above is a norm under which $\cH^p_\D[0,T]$ is a Banach space. Let
$\cH^p[0,T]$ be the set of all pairs $(Y(\cd),Z(\cd\,,\cd))$ of processes having the
following properties: $Y:[0,T]\times\Om\to\dbR^n$ is $\dbF$-adapted, $Z:[0,T]^2
\times\Om\to\dbR^n$ is $\cB([0,T]^2)\otimes\cF_T$-measurable, for each $t\in[0,T]$,
$s\mapsto Z(t,s)$ is $\dbF$-adapted on $[0,T]$, and
$$\|(Y(\cd),Z(\cd\,,\cd))\|_{\cH^p[0,T]}\equiv\Big\{\dbE\[\int_0^T|Y(t)|^pdt
+\int_0^T\(\int_0^T|Z(t,s)|^2ds\)^{p\over2}dt\]\Big\}^{1\over p}<\infty.$$
The above is also a norm under which $\cH^p[0,T]$ is a Banach space. Note that the difference between $\cH^p_\D[0,T]$ and $\cH^p[0,T]$ is that for any $(Y(\cd),Z(\cd\,,\cd))$, if it is in the former, then $Z(\cd\,,\cd)$ is defined only on the upper triangle domain $\D[0,T]$, whereas if it is in the latter, $Z(\cd\,,\cd)$ is defined on the square domain $[0,T]^2$. Next, let
$\cM^p[0,T]$ be the subspace of $\cH^p[0,T]$, consisting of pairs $(Y(\cd),
Z(\cd\,,\cd))$ satisfying the following additional condition:
\bel{M-1}Y(t)=\dbE[Y(t)]+\int_0^tZ(t,s)dW(s),\qq\ae\,t\in[0,T],~\as\ee
Note that for any $(Y(\cd),Z(\cd\,,\cd))\in\cM^p[0,T]$, one has
$$\dbE\(\int_0^t|Z(t,s)|^2ds\)^{p\over2}\les K_p\dbE\Big|\int_0^tZ(t,s)dW(s)\Big|^p
=K_p\dbE|Y(t)-\dbE[Y(t)]|^p\les2^{p-1}K_p\dbE|Y(t)|^p.$$
In the above, the first inequality is called a {\it martingale moment inequality}
(see \autoref{Martingale Moment} in Appendix), with $K_p$ only depending on $p>1$, and $K_2=1$. Thus, for any $(Y(\cd),Z(\cd\,,\cd))\in\cM^p[0,T]$,
\begin{align}
&\|(Y(\cd),Z(\cd\,,\cd))\|_{\cH^p[0,T]}^p=\dbE\[\int_0^T|Y(t)|^pdt
+\int_0^T\(\int_0^T|Z(t,s)|^2ds\)^{p\over2}dt\]\nn\\
&\q\les 2^{(p-2)^+\over2}\|(Y(\cd),Z(\cd\,,\cd))\|_{\cH_\D^p[0,T]}^p+2^{(p-2)^+\over2}\dbE\int_0^T\(\int_0^t|Z(t,s)|^2
ds\)^{p\over2}dt\nn\\
&\q\les 2^{(p-2)^+\over2}\|(Y(\cd),Z(\cd\,,\cd))\|_{\cH_\D^p[0,T]}^p+2^{p-1+{(p-2)^+\over2}}K_p
\dbE\int_0^T|Y(t)|^pdt\nn\\
&\q\les K\|(Y(\cd),Z(\cd\,,\cd))\|_{\cH_\D^p[0,T]}^p\les K\|(Y(\cd),
Z(\cd\,,\cd))\|_{\cH^p[0,T]}^p.\label{H-M}
\end{align}
This shows that in $\cM^p[0,T]$, the norms $\|\cd\|_{\cH^p[0,T]}$ and $\|\cd\|_{\cH^p_\D[0,T]}$ are equivalent. We will use $\|\cd\|_{\cM^p[0,T]}$ to denote the induced norm in $\cM^p[0,T]$ (as a subspace of $\cH^p[0,T]$). For any $0\les R<S\les T$,
$\D[R,S]$, $\cH^p_\D[S,T]$, $\cH^p[S,T]$, and $\sM^p[R,S]$ can be defined similarly.
In the following,  we collect some other spaces for the convenience of the reader.
\begin{align*}
C([R,S];\dbR^n)
&\ts=\Big\{\f:[R,S]\to\dbR^n~|~\hb{$\f(\cd)$ is continuous, }\sup_{s\in[R,S]}|\f(s)| <\infty\Big\},\\
L^p_{\cF_T}(R,S;\dbR^n)
&\ts=\Big\{\f:[R,S]\times\Om\to\dbR^n~|~\hb{$\f(\cd)$ is $\cB([R,S])\otimes\cF_T$-measurable, }\\
&\qq\dbE\int_R^S|\f(s)|^p ds<\infty\Big\},\\
L^p_{\dbF}(\Om;L^q(R,S;\dbR^n))
&\ts=\Big\{\f:[R,S]\times\Om\to\dbR^n~|~\hb{$\f(\cd)$ is $\dbF$-progressively measurable, }\\
&\qq \dbE\(\int_R^S|\f(s)|^q ds\)^{p\over q}<\infty\Big\},\\
L^p_{\dbF}(R,S;\dbR^n)
&\ts=L^p_{\dbF}(\Om;L^p(R,S;\dbR^n)),\\
L_\dbF^p(\Om;C([R,S];\dbR^n))
&\ts=\Big\{\f(\cd)\in L^p_{\dbF}(R,S;\dbR^n)~|~\f(\cd)~\hb{is continuous,  }\dbE\big[\ds\sup_{R\les s\les S}|\f(s)|^p\big]<\infty \Big\},\\
L^p_{\dbF}(\Om;L^q(\D[R,S];\dbR^n))
&\ts=\Big\{\f:\1n\D[R,S]\1n\times\1n\Om\1n\to\1n\dbR^n~|~\f(t,\cd)\in L^p_{\dbF}(\Om;L^q(t,S;\dbR^n)), \, \ae~t\1n\in\1n[R,S],\\
&\ts\hp{=\Big\{\ }\qq\qq\qq\qq\qq \sup_{t\in[R,S]}\dbE\(\int_t^S|\f(t,s)|^qds\)^{p\over q}<\infty \Big\}.
\end{align*}

\ms

Next, we recall the following definition.

\bde{solution} \rm
(i) $(Y(\cd),Z(\cd\,,\cd))\in\cH^p_\D[0,T]$ is called an {\it adapted solution}
of Type-I BSVIE \rf{BSVIE-I} if \rf{BSVIE-I} is satisfied in the usual
It\^o's sense.

\ms

(ii) $(Y(\cd),Z(\cd\,,\cd))\in\cM^p[0,T]$ is called an {\it adapted M-solution} of Type-II
BSVIE \rf{BSVIE-II} if \rf{BSVIE-II} is satisfied in the usual It\^o's sense.

\ede

Note that for an adapted solution $(Y(\cd),Z(\cd\,,\cd))$, $Z(\cd\,,\cd)$ is defined
on the upper triangle domain $\D[0,T]$ only. Whereas, for an adapted M-solution $(Y(\cd),Z(\cd\,,\cd))$,
$Z(\cd\,,\cd)$ is defined on the square domain $[0,T]^2$ and \rf{M-1} is satisfied.

\ms

Now, we recall the standard approach to BSVIEs with adapted generators, and will reveal why the method is not directly applicable to the case of anticipating generators. For the simplicity,
we only consider the Type-I BSVIEs with $p=2$.
Let us now introduce the following hypothesis.

\ms

{\bf(H2)$_2$} The map $g:\D[0,T]\times\dbR^n\times\dbR^n\times\Om\to\dbR^n$
is $\sB(\D[0,T]\times\dbR^n\times\dbR^n)\otimes\sF_T$-measurable, and for each $(t,y,z)\in
[0,T]\times\dbR^n\times\dbR^n$, the process
$s\mapsto g(t,s,y,z)$ is $\dbF$-progressively measurable. Further, there are deterministic
measurable functions $L_y,L_z:\D[0,T]\to[0,\infty)$ having properties that for
some $\e>0$,
$$\int_0^T\(\int_t^TL_y(t,s)^2ds\)^{1+\e}dt+\sup_{t\in[0,T]}\int_t^TL_z(t,s)^{2+\e}ds
<\infty,$$
such that
$$\ba{ll}
\ns\ds|g(t,s,y,z)|\les|g(t,s,0,0)|+L_y(t,s)|y|+L_z(t,s)|z|,
\q(t,s)\in\D[0,T],~y,z\in\dbR^n,\\
\ns\ds|g(t,s,y_1,z_1)\1n-\1n g(t,s,y_2,z_2)|\1n\les\1n L_y(t,s)|y_1\1n-\1n y_2|\1n+\1n L_z(t,s)
|z_1\1n-\1n z_2|,\q(t,s)\in\D[0,T],~y_i,z_i\in\dbR^n,~i=1,2.\ea$$
Also,
$$\dbE\int_0^T\(\int_t^T|g(t,s,0,0,0)|ds\)^2dt<\infty.$$

\ms

Below, we will denote
$$L_0(t,s)=|g(t,s,0,0,0)|,\qq\forall(t,s)\in\D[0,T].$$
We now consider the well-posedness of Type-I BSVIE \rf{BSVIE-I}, under (H2)$_2$.

\ms

Let a pair $(y(\cd),z(\cd\,,\cd))\in\sH_\D^2[0,T]$ and $\psi(\cd)\in
L^2_{\sF_T}(0,T;\dbR^n)$ be given. For almost every $t\in[0,T]$ with
$\psi(t)$ and $g(t,\cd\,,\cd\,,\cd)$ being fixed,
we consider the following BSDE for $(\eta(t,\cd),\z(t,\cd))$:
\bel{BSVIE2}\eta(t,r)=\psi(t)+\int_r^Tg(t,s,y(s),\z(t,s))ds
-\int_r^T\z(t,s)dW(s),\q r\in[t,T],\ee
with $t\in[0,T)$ being a parameter. Clearly, the above is a BSDE (on $[t,T]$) with
adapted  generator $(s,\eta,\z)\mapsto\wt g(s,\eta,\z)\equiv g(t,s,y(s),\z)$
(independent of $\eta$) satisfying
$$\2n\ba{ll}
\ds|\wt g(s,0,0)|=|g(t,s,y(s),0)|\les L_0(t,s)+L_y(t,s)|y(s)|,\qq s\in(t,T],\\
\ns\ds|\wt g(s,\eta_1,\z_1)-\wt g(s,\eta_2,\z_2)|\les L_z(t,s)|\z_1-\z_2|,\qq
s\in(t,T],~\eta_i,\z_i\in\dbR^n,~i=1,2.\ea$$
Thus,
$$\ba{ll}
\ns\ds\dbE\(\int_t^T|\wt g(s,0,0)|ds\)^2\les2\dbE\[\(\int_t^TL_0(t,s)ds\)^2+\(\int_t^TL_y(t,s)|y(s)|ds\)^2\]\\
\ns\ds\q\les2\dbE\[\(\int_t^T\3n L_0(t,s)ds\)^2\3n+\1n\(\int_t^T\3n L_y(t,s)^2ds\)
\(\int_t^T\3n|y(s)|^2ds
\)\]<\infty.\ea$$
Then, BSDE \rf{BSVIE2} admits a unique adapted solution $(\eta(t,\cd),\z(t,\cd))
\in L^2_\dbF(\Om;C([t,T];\dbR^n))\times L^2_\dbF(t,T;\dbR^n)$ and
the following estimate holds:
$$\dbE\[\sup_{r\in[t,T]}|\eta(t,r)|^2+\int_t^T|\z(t,s)|^2ds\]
\les K\dbE\[|\psi(t)|^2+\(\int_t^T|g(t,s,y(s),0)|ds\)^2\].$$
Set
$$Y(t)=\eta(t,t),\q t\in[0,T],\qq Z(t,s)=\z(t,s),\q(t,s)\in\D[0,T].$$
Then
$$Y(t)=\psi(t)+\int_t^Tg(t,s,y(s),Z(t,s))ds-\int_t^TZ(t,s)dW(s),
\qq\ae\;t\in[0,T],$$
and
$$\ba{ll}
\ns\ds\dbE\[|Y(t)|^2+\int_t^T|Z(t,s)|^2ds\]\les K\dbE\[|\psi(t)|^2+\(\int_t^T|
g(t,s,y(s),0)|ds\)^2\]\\
\ns\ds\q\les K\dbE\Big\{|\psi(t)|^2+\[\int_t^T\(L_0(t,s)+L_y(t,s)|y(s)|\)ds\]^2\Big\},\qq\ae\;t\in[0,T].\ea$$
Consequently,
$$\ba{ll}
\ns\ds\dbE\[\int_0^T|Y(t)|^2dt+\int_0^T\int_t^T|Z(t,s)|^2dsdt\]\\
\ns\ds\q\les K\dbE\[\int_0^T|\psi(t)|^2dt+\int_0^T\(\int_t^TL_0(t,s)ds\)^2dt
+\int_0^T\int_t^TL_y(t,s)^2ds\int_t^T|y(s)|^2dsdt\].\ea$$
This means that the map $(y(\cd),z(\cd\,,\cd))\mapsto(Y(\cd),Z(\cd\,,\cd))$ is well-defined
from $\sH_\D^2[0,T]$ to itself.

\ms

Now, for given $(y_i(\cd),z_i(\cd\,,\cd))\in\cH_\D^2[0,T]$, let $(\eta_i(t,\cd),\z_i(t,
\cd))$ be the corresponding adapted solution to the BSDE \rf{BSVIE2}. Then the following stability estimate holds:
$$\ba{ll}
\ns\ds\dbE\[\sup_{r\in[t,T]}|\eta_1(t,r)-\eta_2(t,r)|^2+\int_t^T|\z_1(t,s)
-\z_2(t,s)|^2ds\]\\
\ns\ds\q\les K\dbE\(\int_t^T|g(t,s,y_1(s),\z_1(t,s))-g(t,s,y_2(s),\z_1(t,s))|ds\)^2.\ea$$
Hence, by setting
$$Y_i(t)=\eta_i(t,t),\q Z_i(t,s)=\z_i(t,s),\q(t,s)\in\D[0,T],$$
one has
$$\ba{ll}
\ns\ds\dbE\[|Y_1(t)-Y_2(t)|^2+\int_t^T|Z_1(t,s)-Z_2(t,s)|^2ds\]\\
\ns\ds\q\les K\dbE\(\int_t^T|g(t,s,y_1(s),\z_1(t,s))-g(t,s,y_2(s),\z_1(t,s))|ds\)^2\\
\ns\ds\q\les K\dbE\(\int_t^TL_y(t,s)|y_1(s)-y_2(s)|ds\)^2,\qq\ae\;t\in[0,T].\ea$$
Then for $\d\in(0,T)$, one has
$$\ba{ll}
\ns\ds\dbE\[\int_{T-\d}^T|Y_1(t)-Y_2(t)|^2dt+\int_{T-\d}^T\int_t^T|Z_1(t,s)
-Z_2(t,s)|^2dsdt\]\\
\ns\ds\q\les K\dbE\[\int_{T-\d}^T\(\int_t^TL_y(t,s)^2ds\)\(\int_t^T|y_1(s)-y_2(s)|^2ds\)dt\].\ea$$
This gives that $(y(\cd),z(\cd\,,\cd))\mapsto(Y(\cd),Z(\cd\,,\cd))$ is a contraction
on $\sH_\D^2[T-\d,T]$, for $\d>0$ small enogh. Hence, there exists a unique fixed point on
$\D[T-\d,T]$. This determines the values of $(Y(\cd),Z(\cd\,,\cd))$ as follows:
\bel{YZ}Y(t),\q t\in[T-\d,T],\qq Z(t,s),\q T-\d\les t<s\les T,\ee
where the domain for $Z(\cd\,,\cd)$ is the region $\textcircled{1}$ indicated in the following figure.

\vskip-1cm

\setlength{\unitlength}{.01in}
~~~~~~~~~~~~~~~~~~~~~~~~~~~~~~~~~~~~~~~~~~~~\begin{picture}(230,240)
\put(0,0){\vector(1,0){170}}
\put(0,0){\vector(0,1){170}}
\put(100,100){\line(0,1){50}}
\put(150,0){\line(0,1){150}}
\put(0,100){\line(1,0){100}}
\put(0,150){\line(1,0){150}}
\thicklines
\put(0,0){\color{red}\line(1,1){150}}
\put(120,140){\makebox(0,0){$\textcircled{1}$}}
\put(50,120){\makebox(0,0){$\textcircled{2}$}}
\put(-10,150){\makebox(0,0)[b]{$T$}}
\put(150,-15){\makebox(0,0)[b]{$T$}}
\put(-25,95){\makebox(0,0)[b]{$T-\d$}}
\put(100,-15){\makebox(0,0)[b]{$T-\d$}}
\put(180,-5){\makebox{$t$}}
\put(0,180){\makebox{$s$}}
\end{picture}

\bs

%
\no For $t\in[0,T-\d]$, we consider
\bel{BSVIE4} Y(t)=\psi^\d(t)+\int_t^{T-\d}g(t,s,Y(s),Z(t,s))ds-\int_t^{T-\d}Z(t,s)dW(s),\q t\in[0,T-\d],\ee
where
\bel{F}\psi^\d(t)=\psi(t)+\int_{T-\d}^Tg(t,s,\bar Y(s),Z(t,s))ds-\int_{T-\d}^TZ(t,s)dW(s),\qq t\in[0,T-\d].\ee
Here, we use $\bar Y(\cd)$ to indicate the already determined $Y(\cd)$ (as in \rf{YZ}), and (in the region $\textcircled{2}$ of the above figure)
$Z(\cd\,,\cd)$ is still unknown yet. Thus, \rf{F} is actually a stochastic Fredholm integral equation for $(\psi^\d(\cd),Z(\cd\,,\cd))$ and we need to find $\psi^\d(\cd)$ which is $\sF_{T-\d}$-measurable so that \rf{BSVIE4} is a BSVIE on $[0,T-\d]$. To solve \rf{F}, we introduce the following BSDE:
\bel{}\eta(t,r)=\psi(t)+\int_r^Tg(t,s,\bar Y(s),\z(t,s))ds-\int_r^T\z(t,s)dW(s),\qq r\in[t\vee(T-\d),T].\ee
This is very similar to \rf{BSVIE2}, and it admits a unique adapted solution $(\eta(t,\cd),\z(t,\cd))$. Now, we set
$$\psi^\d(t)=\eta(t,T-\d),\q Z(t,s)=\z(t,s),\qq(t,s)\in[0,T-\d]\times[T-\d,T].$$
Then, $\psi^\d(t)$ is $\sF_{T-\d}$-measurable and $Z(t,s)$ is determined for $(t,s)
\in[0,T-\d]\times[T-\d,T]$ (see region $\textcircled{2}$ indicated in the above figure).
Since $g(\cd)$ is $\dbF$-adapted, \rf{BSVIE4} is a BSVIE on $[0,T-\d]$ with an
$\dbF$-adapted generator. Then by induction, we can solve \rf{BSVIE-II}
over $[T-\d,T]$, over $[T-2\d,T-\d]$, etc., and eventually over $[0,T]$,
with the adapted solution $(Y(\cd),Z(\cd\,,\cd))$ in the space $\sH_\D^2[0,T]$.

\ms

However, such an approach does not work for the case that $g(\cd)$ is merely $\sF_T$-measurable, not $\dbF$-adapted.
For such a case, $s\mapsto g(t,s,y,z)$ is not necessarily $\sF_{T-\d}$-measurable
and thus \rf{BSVIE4} is not a BSVIE on $[0,T-\d]$. As a matter of fact, since the sum
of the first two terms on the right-hand of \rf{BSVIE4} is $\sF_T$-measurable in
general, the last term on the right-hand side which is an It\^o's integral over
$[0,T-\d]$ will not be enough to make the sum of the three terms on the right-hand
side $\dbF$-adapted. Hence, the approach for BSVIEs with $\dbF$-adapted generators
does not work directly for BSVISs with $\sF_T$-measurable generators;
in particular, unlike the adapted generator cases, solving stochastic Fredholm integral equation \rf{F} is not helpful.

\section{BSVIEs with Anticipating Generators}\label{Sec:Well-posedness}

We now would like to improve the above approach  and introduce some additional
conditions so that the well-posedness of BSVIEs with suitable anticipating
generators can be established.
It turns out that the main part of this approach for Type-I and Type-II BSVIEs are essentially the same. Therefore, we directly discuss the case of Type-II BSVIEs, indicating some differences in the process.

\ms

Let $g:\D[0,T]\times\dbR^n\times\dbR^n\times\dbR^n\times\Om\to\dbR^n$ be a generator of a Type-II BSVIE, which is $\sF_T$-measurable, i.e., the map $s\mapsto g(t,s,y,z,\h z)$ is $\cB([t,T])\otimes\cF_T$-measurable. We introduce the following:
\bel{decomposition}\ba{ll}
\ns\ds g_1(t,s,y,z,\h z)=\dbE_s[g(t,s,y,z,\h z)],\\
\ns\ds g_0(t,s,y,z,\h z)=g(t,s,y,z,\h z)-g_1(t,s,y,z,\h z),\\
\ns\ds\qq\q \forall(t,s,y,z,\h z)\in\D[0,T]\times\dbR^n\times\dbR^n\times\dbR^n.
\ea\ee
We can show, by using Lemma \ref{7.4} in the appendix, that one can take $g_1:\D[0,T]\times\dbR^n\times\dbR^n\times\dbR^n
\times\Om\to\dbR^n$ as an adapted generator. In other
words, from now on, we may assume that the following decomposition holds:
\bel{g=g_0+g_1}g(t,s,y,z,\h z)=g_0(t,s,y,z,\h z)+g_1(t,s,y,z,\h z),
\q\forall(t,s,y,z,\h z)\in\D[0,T]\times\dbR^n\times\dbR^n\times\dbR^n,~\as,\ee
where $g_1(\cd)$ is $\dbF$-adapted and $g_0(\cd)$ is $\sF_T$-measurable.
In the case that the generator is $\dbF$-adapted, $g_0(\cd)$ vanishes.

\ms

Let us rewrite our Type-II BSVIE as follows:
\begin{align}
Y(t)&=\psi(t)+\int_t^Tg_0(t,s,Y(s),Z(t,s),Z(s,t))ds+\int_t^Tg_1(t,s,Y(s),Z(t,s),Z(s,t))ds\nn\\
&\q-\int_t^TZ(t,s)dW(s),\qq t\in[0,T].\label{ABSVIE01}
\end{align}
The main idea is that we treat the sum of the first two terms on the right-hand side of \rf{ABSVIE01} as the free term. Then by some improved estimates for BSVIEs with adapted generators, we establish the well-posedness of BSVIEs with anticipating generators.

\ms

We now introduce the following hypothesis which is needed for Type-II BSVIEs.

\ms

{\bf(H3)$_p$} Let $g_0,g_1:\D[0,T]\times\dbR^n\times\dbR^n\times\dbR^n\times\Om\to\dbR^n$ be
$\sB(\D[0,T]\times\dbR^n\times\dbR^n)\otimes\sF_T$-measurable. For any $(t,y,z,\h z)\in[0,T]\times\dbR^n\times\dbR^n\times\dbR^n$, $s\mapsto g_1(t,s,y,z,\h z)$ is $\dbF$-progressively measurable. There are deterministic functions $L^i_y,L^i_z:
\D[0,T]\to[0,\infty)$, $i=0,1$, with the properties that
\bel{3.3}\ba{ll}
\ns\ds\int_0^T\(\int_t^TL^i_y(t,s)^{(p\land2)(1+\e)\over (p\land2)-1}ds\)^{(p\land2)-1}dt+\sup_{t\in[0,T]}\(\int_t^TL^i_z(t,s)^2ds\)\\
\ns\ds\qq\qq+\sup_{t\in[0,T]}
\(\int_t^TL^i_{\h z}(t,s)^{(p\land2)(1+\e)\over(p\land2)-1}ds\)<\infty,\qq i=0,1,\ea\ee
for some $\e>0$, and with the constant $K_p^0>0$ appearing in \rf{9.13},
\bel{K<1}K^0_p\sup_{t\in[0,T]}\(\int_t^TL^0_z(t,s)^2ds\)^{p\over2}<1,\ee
such that
\begin{align}
|g_i(t,s,y_0,z_0,\h z_0)-g_i(t,s,y_1,z_1,\h z_1)|\les L^i_y(t,s)|y_0-y_1|
\1n+\1n L^i_z(t,s)|z_0-z_1|\1n+\1n L_{\h z}^i(t,s)|\h z_0-\h z_1|,\nn\\
(t,s)\in\D[0,T],~y_j,z_j,\h z_j\in\dbR^n,~i,j=0,1,\label{|g-g|}
\end{align}
and
\bel{|g|}\ba{ll}
\ns\ds|g_i(t,s,y,z,\h z)|\1n\les\1n L^i_0(t,s)\1n+\1n L_y^i(t,s)|y|\1n+\1n L_z^i(t,s)|z|\1n+\1n L^i_{\h z}(t,s)|\h z|^{1\wedge{2\over p}},\q(t,s)\in\D[0,T],~y,z,\h z\in\dbR^n,\ea\ee
where
\bel{L_0}L^i_0(\cd\,,\cd)=|g_i(\cd\,,\cd\,,0,0,0)|,\qq\dbE\int_0^T\(\int_t^TL^i_0(t,s)
ds\)^pdt<\infty,\qq i=0,1.\ee

\ms

Note that condition \rf{K<1} means that the dependence $z\mapsto g_0(t,s,y,z,\h z)$ is well-controlled.
Since in the appendix, we have precise estimate for the constant $K_p^0$, this condition just gives a size restriction of $L_z^0(\cd\,,\cd)$, which need not be sufficiently small. This condition will play a crucial role below. We believe that a restriction on the dependence of $z\mapsto g_0(t,s,y,z,\h z)$ is somewhat necessary;
also see \autoref{remark-fbsde} for a comment on this condition from the viewpoint of FBSDEs.
However, at the moment, we can not provide satisfying evidence. Hopefully, we can address this point in our future research.

\ms

For Type-I BSVIEs, we impose the following  hypothesis.

\bs

{\bf(H3)$_p'$ \rm In (H3)$_p$, $L_{\h z}^i(\cd\,,\cd)=0$ and \rf{3.3} is replaced by the following:
\bel{3.3*}\ba{ll}
\ns\ds\int_0^T\(\int_t^TL^i_y(t,s)^{p(1+\e)\over p-1}ds\)^{p-1}dt+\sup_{t\in[0,T]}\(\int_t^TL^i_z(t,s)^{2}ds\)<\infty,\q i=0,1.\ea\ee

\ms

Note that for $p>2$, one has ${p\over p-1}<2$. Thus, for such a case, condition \rf{3.3} is stronger than \rf{3.3*} for $L_y(\cd\,,\cd)$.
We also note that in the case $g_0(\cd)$ being independent of $z$, condition \rf{K<1} is
automatically true. In particular, such a condition holds true
if $g_0(\cd)$ is identically equal to zero (or the generator of BSVIE is adapted). In what follows, we denote
$$\left\{\2n\ba{ll}
\ds L_0(t,s)=|g_0(t,s,0,0,0)|+|g_1(t,s,0,0,0)|,\\
\ns\ds L_y(t,s)=L_y^0(t,s)+L^1_y(t,s),\q L_{\h z}(t,s)=L^0_{\h z}(t,s)+L^0_{\h z}(t,s),\ea\qq(t,s)\in\D[0,T].\right.$$

\bt{Well-posed-ABSVIE} \sl Let {\rm(H3)$_p$} hold. Then, for any $\psi(\cd)\in
L^p_{\sF_T}(\Om;\dbR^n)$, Type-II BSVIE \rf{ABSVIE01} admits a unique adapted M-solution $(Y(\cd),Z(\cd\,,\cd))\in\sM^p[0,T]$. Moreover,
\bel{BSVIE-estimate1}\dbE\[\int_0^T|Y(t)|^pdt+\int_0^T\(\int_t^T|Z(t,s)|^2
ds\)^{p\over2}dt\]\les K\dbE\[\int_0^T|\psi(t)|^pdt+\int_0^T\(\int_t^TL_0(t,s)ds
\)^pdt\].\ee
If $\bar g(\cd\,,\cd\,,\cd\,,\cd\,,\cd)$ is another generator satisfying {\rm(H3)$_p$}, $\bar\psi(\cd)\in L^p_{\cF_T}(\Om;\dbR^n)$, and $\bar Y(\cd),\bar Z(\cd\,,\cd))\in\cM^p[0,T]$ is the unique adapted M-solution to the corresponding BSVIE, then
\begin{align}
&\dbE\[\int_0^T|Y(t)-\bar Y(t)|^pdt+\int_0^T\(\int_t^T|Z(t,s)-\bar Z(t,s)|^2ds\)^{p\over2}dt\les\1n K\dbE\[\1n\int_0^T\3n|\psi(t)\1n-\1n\bar\psi(t)|^pdt\nn\\
&\qq+\2n\int_0^T\3n\1n\(\1n\int_t^T\3n
|g(t,s,Y(s),Z(t,s),Z(s,t))\1n-\1n\bar g(t,s,Y(s),Z(t,s),Z(s,t))|ds\)^p\1n dt\].\label{BSVIE-estimate2}
\end{align}
\et

\begin{proof} Let $(y(\cd),z(\cd\,,\cd))\in\sM^p[0,T]$. We consider the following Type-I BSVIE:
\begin{align}
 Y(t)&=\psi(t)+\int_t^T\[g_0(t,s,y(s),z(t,s),z(s,t))+g_1(t,s,y(s),0,z(s,t))\]ds\nn\\
&\q+\int_t^T\[g_1(t,s,y(s),Z(t,s),z(s,t))-g_1(t,s,y(s),0,z(s,t))\]ds\nn\\
&\q-\int_t^TZ(t,s)dW(s),\qq t\in[0,T].\label{ABSVIE1*}
\end{align}
The generator is
$$\wt g(t,s,\z)=g_1(t,s,y(s),\z,z(s,t))-g_1(t,s,y(s),0,z(s,t)),\qq(t,s,\z)\in
\D[0,T]\times\dbR^n,$$
with the free term
$$\wt\psi(t)=\psi(t)+\int_t^T\[g_0(t,s,y(s),z(t,s),z(s,t))+g_1(t,s,y(s),0,z(s,t))\]ds.$$
Clearly,
$$\ba{ll}
\ns\ds\wt g(t,s,0)=0,\qq(t,s)\in\D[0,T],\\
\ns\ds|\wt g(t,s,\z_1)-\wt g(t,s,\z_2)|\les L_z^1(t,s)|\z_1-\z_2|,\q(t,s)\in\D[0,T],~\z_1,\z_2\in\dbR^n.\ea$$
Thus, BSVIE \rf{ABSVIE1*} admits a unique adapted solution $(Y(\cd),Z(\cd\,,\cd))\in\cH^p_\D[0,T]$. We naturally extend $Z(\cd\,,\cd)$ from $\D[0,T]$ to $[0,T]^2$ via the following
\bel{M}Y(t)=\dbE [Y(t)]+\int_0^tZ(t,s)dW(s),\qq t\in[0,T].\ee
Then, by \rf{H-M} we have $(Y(\cd),Z(\cd\,,\cd))\in\cM^p[0,T]$.
This means a solution map $(y(\cd),z(\cd\,,\cd))\mapsto(Y(\cd),Z(\cd\,,\cd))$ is well-defined from $\cM^p[0,T]$ to itself.

\ms

Now, let $(y_i(\cd),z_i(\cd\,,\cd))\in\cM^p[0,T]$ and let $(Y_i(\cd),Z_i(\cd\,,\cd))\in\cM^p[0,T]$ be the adapted M-solution to
BSVIE \rf{ABSVIE1*} with $(y(\cd),z(\cd\,,\cd))$ replaced by $(y_i(\cd),z_i(\cd\,,\cd))$. Then
\begin{align}
 Y_1(t)-Y_2(t)&=\int_t^T\[g_0(t,s,y_1(s),z_1(t,s),z_1(s,t))-g_0(t,s,y_2(s),z_2(t,s),
z_2(s,t))\nn\\
&\qq\qq+g_1(t,s,y_1(s),Z_1(t,s),z_1(s,t))-g_1(t,s,y_2(s),Z_1(t,s),
z_2(s,t))\]ds\nn\\
&\q+\int_t^TG(t,s)[Z_1(t,s)-Z_2(t,s)]ds-\int_t^T[Z_1(t,s)-Z_2(t,s)]dW(s),\q t\in[0,T],\label{ABSVIE2}
\end{align}
where
$$\ba{ll}
\ns\ds G(t,s)={1\over|Z_1(t,s)-Z_2(t,s)|^2}\[g_1(t,s,y_2(s),Z_1(t,s),z_2(s,t))\\
\ns\ds\qq\qq-g_1(t,s,y_2(s),Z_2(t,s),z_2(s,t))\][Z_1(t,s)-Z_2(t,s)]^\top{\bf1}_{\{Z_1(t,s)
\ne Z_2(t,s)\}}.\ea$$
Clearly,
$$|G(t,s)|\les L_z^1(t,s),\qq(t,s)\in\D[0,T].$$
By \autoref{simple BSVIE}, we see that BSVIE \rf{ABSVIE2} admits a unique adapted solution
$(Y_1(\cd)-Y_2(\cd),Z_1(\cd\,,\cd)-Z_2(\cd\,,\cd))\in\cM^p[0,T]$. Moreover,
$$\ba{ll}
\ns\ds\dbE|Y_1(t)-Y_2(t)|^p+\dbE\(\int_t^T|Z_1(t,s)-Z_2(t,s)|^2ds\)^{p\over2}\\
\ns\ds\q\les K_p^0\dbE\Big|\int_t^T\big[g_0(t,s,y_1(s),z_1(t,s),z_1(s,t))-g_0(t,s,y_2(s),z_2(t,s),
z_2(s,t))\\
\ns\ds\qq\qq\qq+g_1(t,s,y_1(s),Z_1(t,s),z_1(s,t))-g_1(t,s,y_2(s),Z_1(t,s),z_2(s,t))\big]ds\Big|^p\\
\ns\ds\q\les K_p^0\dbE\(\int_t^T\big[L_y(t,s)|y_1(s)-y_2(s)|+L_z^0(t,s)|z_1(t,s)-z_2(t,s)|
+L_{\h z}(t,s)|z_1(s,t)-z_2(s,t)|\big]ds\)^p\\
\ns\ds\q\les(1+\e_0)K_p^0\dbE\(\int_t^TL^0_z(t,s)|z_1(t,s)-z_2(t,s)|ds\)^p+K\dbE\(
\int_t^TL_y(t,s)|y_1(s)-y_2(s)|ds\)^p\\
\ns\ds\qq+K\dbE\(\int_t^TL_{\h z}(t,s)|z_1(t,s)
-z_2(t,s)|ds\)^p,\qq\ae\;t\in[0.T].\ea$$
Now, for any $\beta>0$, from the above, we have
$$\ba{ll}
\ns\ds\dbE\[\int_0^Te^{\beta pt}|Y_1(t)-Y_2(t)|^pdt+\int_0^Te^{\beta pt}\(\int_t^T|Z_1(t,s)-Z_2(t,s)|^2ds\)^{p\over2}dt\]\\
\ns\ds\q\les(1+\e_0)K_p^0\dbE\int_0^Te^{\beta pt}\(\int_t^TL^0_z(t,s)|z_1(t,s)-z_2(t,s)|ds\)^pdt\\
\ns\ds\qq+K\dbE\int_0^Te^{\beta pt}\(\int_t^TL_y(t,s)|y_1(s)-y_2(s)|ds\)^pdt\\
\ns\ds\qq+K\dbE\int_0^Te^{\beta pt}\(\int_t^TL_{\h z}(t,s)|z_1(s,t)-z_2(s,t)|ds\)^pdt.\ea$$
Note that
$$\ba{ll}
\ns\ds\dbE\int_0^Te^{\beta pt}\(\int_t^TL^0_z(t,s)|z_1(t,s)-z_2(t,s)|ds\)^pdt\\
\ns\ds\q=\dbE\int_0^T\(\int_t^TL^0_z(t,s)^2ds\)^{p\over2}e^{\beta pt}\(\int_t^T|z_1(t,s)-z_2(t,s)|^2ds\)^{p\over2}dt\\
\ns\ds\q\les\[\sup_{t\in[0,T]}\(\int_t^TL^0_z(t,s)^2ds\)^{p\over2}\]\dbE\int_0^Te^{\beta pt}\(\int_t^T|z_1(t,s)-z_2(t,s)|^2ds\)^{p\over2}dt,\ea$$
and
$$\ba{ll}
\ns\ds\dbE\int_0^Te^{\beta pt}\(\int_t^TL_y(t,s)|y_1(s)-y_2(s)|ds\)^pdt\\
\ns\ds\q\les\dbE\int_0^T\(\int_t^Te^{-{\beta p(s-t)\over p-1}}L_y(t,s)^{p\over p-1}ds\)^{p-1}\(\int_t^Te^{\beta p s}|y_1(s)-y_2(s)|^pds\)dt\\
\ns\ds\q\les\int_0^T\(\int_t^Te^{-{\beta p(1+\e)(s-t)\over (p-1)\e}}ds\)^{(p-1)\e\over1+\e}\(\int_t^TL_y(t,s)^{p(1+\e)\over p-1}ds\)^{p-1\over1+\e}dt\,
\dbE\int_0^Te^{\beta p s}|y_1(s)-y_2(s)|^pds\\
\ns\ds\q\les\({(p-1)\e\over\beta p(1+\e)}\)^{(p-1)\e\over1+\e}\int_0^T\(\int_t^TL_y(t,s)^{p(1+\e)\over p-1}ds\)^{p-1\over1+\e}dt\, \dbE\int_0^Te^{\beta p s}|y_1(s)-y_2(s)|^pds.\ea$$
Also, since
$$y_1(t)-y_2(t)=\dbE[y_1(t)-y_2(t)]+\int_0^t[z_1(t,s)-z_2(t,s)]dW(s),\q t\in[0,T],$$
one has
$$\dbE\(\int_0^t|z_1(t,s)-z_2(t,s)|^2ds\)^{p\over2}\les K_p\dbE\Big|\int_0^t[z_1(t,s)-z_2(t,s)]dW(s)\Big|^p\les2^{p-1}K_p\dbE|y_1(t)-y_2(t)|^p.$$
Therefore, when $p\in(1,2]$,
\begin{align*}
&\dbE\int_0^Te^{\beta pt}\(\int_t^TL_{\h z}(t,s)|z_1(s,t)-z_2(s,t)|ds\)^pdt\\
&\q=\dbE\int_0^T\(\int_t^Te^{\beta(t-s)}L_{\h z}(t,s)e^{\beta s}|z_1(s,t)-z_2(s,t)|ds\)^pdt\\
&\q\les\dbE\int_0^T\(\int_t^Te^{{\beta p\over p-1}(t-s)}L_{\h z}(t,s)^{p\over p-1}ds\)^{p-1}\(\int_t^Te^{\beta ps}|z_1(s,t)-z_2(s,t)|^pds\)dt\\
&\q\les\dbE\int_0^T\(\int_t^Te^{{\beta p(1+\e)\over(p-1)\e}(t-s)}ds\)^{(p-1)\e\over1+\e}\(\int_t^TL_{\h z}(t,s)^{p(1+\e)\over (p-1)}ds\)^{(p-1)\over1+\e}\(\int_t^Te^{\beta ps}|z_1(s,t)-z_2(s,t)|^pds\)dt\\
&\q\les\({(p-1)\e\over\beta p(1+\e)}\)^{(p-1)\e\over1+\e}
\[\sup_{t\in[0,T]}\(\int_t^TL_{\h z}(t,s)^{p(1+\e)\over(p-1)}ds\)^{(p-1)\over1+\e}\]
\dbE\int_0^T\int_t^Te^{\beta ps}|z_1(s,t)-z_2(s,t)|^pdsdt\\
&\q\les\({(p-1)\e\over\beta p(1+\e)}\)^{(p-1)\e\over1+\e}
\[\sup_{t\in[0,T]}\(\int_t^TL_{\h z}(t,s)^{p(1+\e)\over(p-1)}ds\)^{(p-1)\over1+\e}\]\dbE\int_0^Te^{\beta pt}
\int_0^t|z_1(t,s)-z_2(t,s)|^pdsdt\\
&\q\les\({(p-1)\e\over\beta p(1+\e)}\)^{(p-1)\e\over1+\e}T^{1-{p\over2}}
\[\sup_{t\in[0,T]}\(\int_t^TL_{\h z}(t,s)^{p(1+\e)\over(p-1)}ds\)^{(p-1)\over1+\e}\]\dbE\int_0^Te^{\beta pt}
\(\int_0^t|z_1(t,s)-z_2(t,s)|^2ds\)^{p\over2}dt\\
%
%
&\q\les\({(p-1)\e\over\beta p(1+\e)}\)^{(p-1)\e\over1+\e}T^{1-{p\over2}}
2^{p-1}K_p\[\sup_{t\in[0,T]}\(\int_t^TL_{\h z}(t,s)^{p(1+\e)\over(p-1)}ds\)^{(p-1)\over1+\e}\]\dbE\int_0^Te^{\beta pt}|y_1(t)-y_2(t)|^pdt.\end{align*}
Summarizing the above, we obtain that when $p\in(1,2]$, the following holds:
$$\ba{ll}
\ns\ds\dbE\[\int_0^Te^{\beta pt}|Y_1(t)-Y_2(t)|^pdt+\int_0^Te^{\beta pt}\(\int_t^T|Z_1(t,s)-Z_2(t,s)|^2ds\)^{p\over2}dt\]\\
\ns\ds\q\les(1+\e_0)K_p^0\[\sup_{t\in[0,T]}\(\int_t^TL^0_z(t,s)^2ds\)^{p\over2}\]\dbE
\int_0^Te^{\b pt}\(\int_t^T|z_1(t,s)-z_2(t,s)|^2ds\)^{p\over2}dt\\
\ns\ds\qq+\({(p-1)\e\over\beta p(1+\e)}\)^{(p-1)\e\over1+\e}K\[\int_0^T\(\int_t^TL_y(t,s)^{p(1+\e)\over p-1}ds\)^{p-1\over1+\e}dt\]\dbE\int_0^Te^{\beta p s}|y_1(s)-y_2(s)|^pds\\
\ns\ds\qq+\({(p-1)\e\over\beta p(1+\e)}\)^{(p-1)\e\over1+\e}T^{1-{p\over2}}K
\[\sup_{t\in[0,T]}\(\int_t^TL_{\h z}(t,s)^{p(1+\e)\over(p-1)}ds\)^{(p-1)\over1+\e}\]\dbE\int_0^Te^{\beta pt}|y_1(t)-y_2(t)|^pdt.\ea$$
If $L_{\h z}(\cd\,,\cd)=0$, i.e., our BSVIE is Type-I, then the above holds for any $p\in(1,\infty)$. In this two cases, under the crucial condition \rf{K<1},
for $\beta>0$ large enough, the map $(y(\cd),z(\cd\,,\cd))\mapsto(Y(\cd),Z(\cd\,,\cd))$ is a contraction on $\cH^{p,\beta}_\D[0,T]$ whose norm is given by
$$\|(Y(\cd),Z(\cd\,,\cd)\|_{\cH^{p,\beta}_\D[0,T]}\equiv\Big\{\dbE\[\int_0^Te^{\beta p t}|Y(t)|^pdt+\int_0^Te^{\beta pt}\(\int_t^T|Z(t,s)|^2ds\)^{p\over2}dt\]\Big\}^{1\over p}.$$
Clearly, this norm is equivalent to that of $\cH^p_\D[0,T]$. Hence, BSVIE \rf{ABSVIE01} admits a unique adapted solution $(Y(\cd),Z(\cd\,,\cd))\in\cH^p_\D[0,T]$; and by naturally extending $Z(\cd\,,\cd)$ from $\D[0,T]$ to $[0,T]^2$ as in \rf{M}, one has $(Y(\cd),Z(\cd\,,\cd))\in\cM^p[0,T]$.

\ms

Finally, if our BSVIE is of Type-II, and $p\in[2,\infty)$. For $\psi(\cd)\in L^p_{\cF_T}(0,T;\dbR^m)\subseteq
L^2_{\cF_T}(0,T;\dbR^m)$, thanks to \rf{3.3}, there exists a unique
adapted M-solution $(Y(\cd),Z(\cd\,,\cd))\in\cM^2[0,T]$. We want to
show that in the current case, $(Y(\cd),Z(\cd\,,\cd))$ is actually
in $\cM^p[0,T]$. To show this, for the obtained adapted M-solution
$(Y(\cd),Z(\cd\,,\cd))$, let us consider the following Type-I BSVIE for
$(\wt Y(\cd),\wt Z(\cd\,,\cd))$:
\bel{BSVIE-3.14}
\wt Y(t)=\psi(t)+\int_t^Tg\big(t,s,\wt Y(s),\wt Z(t,s),Z(s,t)\big)ds-\int_t^T\wt Z(t,s)dW(s),\qq t\in[0,T].\ee
Note that
$$\ba{ll}
\ds\dbE\int_0^T\Big(\int_t^T|g(t,s,0,0,Z(s,t))|ds\Big)^pdt\\
\ns\ds\q\les K\dbE\int_0^T\Big(\int_t^T\big[L_0(t,s)+L_{\h z}(t,s)|Z(s,t)|^{2\over p}\big]ds\Big)^pdt\\
\ds\q\les K\dbE\[\int_0^T\(\int_t^TL_0(t,s)ds\)^pdt+\int_0^T\(\int_t^TL_{\h z}(t,s)^{p\over p-1}ds\)^{p-1}\(\int_t^T|Z(s,t)|^2\)dsdt\]\\
\ns\ds\q\les K\dbE\[\int_0^T\(\int_t^TL_0(t,s)ds\)^pdt+\int_0^T\int_0^t|Z(t,s)|^2dsdt\]<\infty.\ea$$
Thus, by the results on Type-I BSVIEs, \rf{BSVIE-3.14} admits a unique adapted solution $(\wt Y(\cd),\wt Z(\cd\,,\cd))\in\cH^p_\D[0,T]$.
Now,  we extend $\wt Z(t,s)$ from $\D[0,T]$ to $[0,T]^2$ similarly to \rf{M}.
 It leads to $(\wt Y(\cd),\wt Z(\cd\,,\cd))\in\cM^p[0,T]\subseteq\cM^2[0,T]$, which is an adapted M-solution to \rf{BSVIE-3.14}. But, $(Y(\cd),Z(\cd\,,\cd))$ is already the unique adapted M-solution to \rf{BSVIE-3.14} in $\cM^2[0,T]$. Hence, it is necessary that
$$(Y(\cd),Z(\cd\,,\cd))=(\wt Y(\cd),\wt Z(\cd\,,\cd))\in\cM^p[0,T].$$
The stability estimates can be obtained similarly.
\end{proof}

\begin{remark}\rm
Let us revisit \autoref{eample1}, in which the following BSDE is considered:
\bel{BSDE-Example-revisit}
Y(t)=\int_t^TW(T)ds-\int_t^TZ(s)dW(s),\q t\in[0,T].
\ee
Denote
$$
\psi(t)=\int_t^TW(T)ds,\q t\in[0,T].
$$
Then we can rewrite \rf{BSDE-Example-revisit} as follows:
\bel{BSDE-Example-revisit1}
Y(t)=\psi(t)-\int_t^TZ(s)dW(s),\q t\in[0,T].
\ee
Note that the new free term $\psi(\cd)$ depends on $t$, though
the generator of \rf{BSDE-Example-revisit} is independent of $t$ and the free term of  \rf{BSDE-Example-revisit} equals zero.
Because of this feature, we should allow the unknown process $Z(t,s)$ to depend on the parameter $t$; that is
\bel{BSDE-Example-revisit2}
Y(t)=\int_t^TW(T)ds-\int_t^TZ(t,s)dW(s),\q t\in[0,T].
\ee
Naturally, the backward equation with the anticipating generator $W(T)$ should be studied  in the framework of BSVIEs,
rather than BSDEs.
\end{remark}

\section{Path-Dependent Type-I BSVIEs}\label{Sec:path-dependent-BSVIE}

In this section, we shall consider the following path-dependent Type-I BSVIE:
\bel{BSVIE-path}
Y(t)=\psi(t)+\int_t^T g(t,s,Y_s)ds-\int_t^TZ(t,s)dW(s),\q t\in[0,T],
\ee
where $Y_s$ denotes the path of $Y(\cd)$ from the current time $s$ to the terminal time $T$, that is
$$Y_s\equiv\{Y(r)\bigm|s\les r\les T\},\q\as$$
Denote $\dbY\deq\bigcup_{s\in[0,T]}\dbY_s$, where $\dbY_s\deq C([s,T];\dbR^n)$, equipped with the uniform norm:
$$\|\By_s\|_{\dbY_s}=\sup_{r\in[s,T]}|\By_s(r)|,\q\forall \By_s\in\dbY_s.$$

\ss
{\bf (H4)$_p$} The map $\psi:[0,T]\times\Om\to\dbR^n$ is $\cF_T$-measurable
and the map $g:\D[0,T]\times \dbY\times\Om\to\dbR^n$ is $\cB(\D[0,T]\times \dbY)\otimes\cF_T$-measurable satisfying:
\bel{}
g(t,s,{\bf y})=g(t,s,{\bf y}_s),\q \forall (t,s,{\bf y})\in\D[0,T]\times\dbY.
\ee
There exist a constant $L>0$, a modulus of continuous function $\rho(\cd)$ and a positive random variable $\xi\in L^p_{\cF_T}(\Om;\dbR_+)$ such that:
\begin{align}
\nn&|g(t,s,\By_s)-g( t,s,\By^\prime_s)|\les L\|\By_s-\By^\prime_s\|_{\dbY_s},\q \int_t^{t+\d}|g(t,s,{\bf 0})|ds\les |\xi|\rho(\d),\\
\nn&|g(t-\d,s,\By_s)-g( t,s,\By_s)|+|\psi(t-\d)-\psi(t)|\les (|\xi|+\|\By_s\|_{\dbY_s})\rho(\d),\\
&\qq\qq\qq\qq\qq\qq\qq\q \forall (t,s)\in\D[0,T],\, \By_s,\By_s^\prime\in\dbY_s,\, \d\in[0,t].
\end{align}
Moreover,
\bel{II}\dbE\[\sup_{t\in[0,T]}\Big|\int_t^T g(t,s,{\bf 0})ds\Big|^p+\sup_{t\in[0,T]}|\psi(t)|^p\]<\infty.\ee

\ms

\begin{theorem}\label{thm:well-posedness-A} \sl Let {\rm(H4)$_p$} hold with $p>1$.
Then BSVIE \rf{BSVIE-path} admits a unique adapted solution $(Y(\cd),Z(\cd\,,\cd))$
such that $Y(\cd)$ is continuous in $t$ and the following estimate holds true:
\bel{BSVIEest}
\dbE\Big[\sup_{0\les t\les T}|Y(t)|^p\Big]\1n+\2n\sup_{0\les t\les T}\dbE\(\int_t^T\3n |Z(t,s)|^2ds\)^{p\over 2}\1n\les\1n K\dbE\[\sup_{t\in[0,T]}\Big|\int_t^T\3n g(t,s,{\bf 0})ds\Big|^p\2n+\2n\sup_{t\in[0,T]}|\psi(t)|^p\].\ee

\end{theorem}

\ms
Note that the generator $g(\cd)$ of BSVIE \rf{BSVIE-path}
could depend on the whole pathes of the unknown process $Y(\cd)$ in the future,
due to which we would like to call \rf{BSVIE-path} a {\it path-dependent BSVIE}.
To our best knowledge, it is the first time to study the backward stochastic equations in a path-dependent framework. It is noteworthy that the difficulty caused by the  $\cF_T$-measurable generator is not avoidable here,
because for any $Y(\cd)\in L^p_{\dbF}(\Om;C([0,T];\dbR^n))$,
$g(t,s,Y_s)$ is still $\cF_T$-measurable even if $g(t,s,\By_s)$ is assumed to be $\cF_s$-measurable for
any deterministic function $\By_s\in C([s,T];\dbR^n)$.

%

\ms
A closely related work of path-dependent BSVIEs is the following so-called {\it anticipated BSDE} (ABSDE, for short), introduced by Peng--Yang \cite{Peng-Yang2009}:
\bel{ABSDE}\left\{\begin{aligned}
dY(t)&=-f(t,Y(t),Y(t+\d),Z(t),Z(t+\d))dt+Z(t)dW(t),\q t\in[0,T],\\
 Y(t)&=\eta(t),\q Z(t)=\zeta(t),\q t\in[T,T+\d],
\end{aligned}\right.\ee
where $0<\d<T$ is a fixed constant and $(\eta(\cd),\zeta(\cd))\in  L^2_{\dbF}(\Om;C([T,T+\d];\dbR^n))\times L^2_\dbF(T,T+\d;\dbR^n)$
is the given terminal state.  Along this line, some extensions of ABSDEs to BSVIEs (called {\it anticipated BSVIEs})
can be found in \cite{Wen-Shi2020,Hamaguchi2021}.
The main feature of \rf{ABSDE} is that the generator $f(\cd)$  at time $t$ depends on the values of the unknown processes
at the future time $t+\d$. Note from Peng--Yang  \cite{Peng-Yang2009} that in studying the well-posedness of ABSDEs, the following
{\it adaptedness condition} is assumed:
\bel{M-condition}
f(s,y(s),y(s+\d),z(s),z(s+\d))\in\cF_s,\q \forall (y(\cd),z(\cd))\in L^2_{\dbF}(\Om;C([0,T+\d];\dbR^n))\times L^2_\dbF(0,T+\d;\dbR^n).
\ee
To satisfy the adaptedness condition \rf{M-condition}, one normally needs to impose the following assumption:
\begin{align}
f(s,y(s),y(s+\d),z(s),z(s+\d))=\ti f(s,y(s),\dbE_s[y(s+\d)],z(s),\dbE_s[z(s+\d)]),\nn\\
\forall (y(\cd),z(\cd))\in L^2_{\dbF}(\Om;C([0,T+\d];\dbR^n))\times L^2_\dbF(0,T+\d;\dbR^n),\label{M-condition1}
\end{align}
for some progressively function $\ti f(\cd)$, from which we see that the anticipated terms $y(s+\d)$ and $z(s+\d)$
have been taken conditional expectation. Mathematically, it is not satisfying,
and  then naturally one hopes to remove this adaptedness condition of ABSDEs.
However, the following example shows that it cannot be achieved within the framework of ABSDEs.
The  feasible approach (see \autoref{thm:well-posedness-A}) is to generalize the framework to BSVIEs.
The main reason is that the generator of BSVIEs is allowed to be anticipating,
whereas for BSDEs, the generator must be adapted.

\begin{example}\rm
Consider the ABSDE:
\bel{example-2}
Y(t)=W(2)+\int_t^2 Y([s+1]\wedge 2)ds-\int_t^2 Z(s)dW(s),\q t\in[0,2].
\ee
Note that
$$
Y([t+1]\wedge 2)=Y(2)=W(2),\q t\in[1,2].
$$
Then on $[1,2]$,
$$
Y(t)=(3-t)W(2)-\int_t^2 Z(s)dW(s),\q t\in[1,2].
$$
Obviously,
$$
\cY(t)=(3-t)W(t),\q \cZ(t,s)=(3-t),\q (t,s)\in\D[1,2],
$$
is the unique solution to the following BSVIE:
$$
\cY(t)=(3-t)W(2)-\int_t^2\cZ(t,s)dW(s),\q t\in[1,2].
$$
Then by the argument in \autoref{eample1}, we obtain that ABSDE \rf{example-2}
dose not have an adapted solution on $[1,2]$.
\end{example}

Next, we are going to prove \autoref{thm:well-posedness-A}.
For any given $y(\cd)\in L_{\dbF}^p(\Om;C([0,T];\dbR^n))$, we consider the following BSVIE:
\bel{BSVIE-path-By}
Y(t)=\psi(t)+\int_t^T g(t,s,y_s)ds-\int_t^TZ(t,s)dW(s),\q t\in[0,T].
\ee
By  modifying the argument employed in \cite[Proposition 2.4]{Wang-Yong-Zhang2021},
we have the following $L_{\dbF}^p(\Om;C([0,T];\dbR^n))$-norm estimate for $Y(\cd)$.

\begin{proposition}\label{Prop:BSVIE-no-y}
Let {\rm(H4)$_p$} hold with $p>1$.
Then  BSVIE \rf{BSVIE-path-By} admits a unique adapted solution $(Y(\cd),Z(\cd,\cd))$
such that $Y(\cd)$ is continuous in $t$ and the following estimate holds true:
\begin{align}
&\dbE\Big[\sup_{0\les t\les T} |Y(t)|^p \Big] + \sup_{0\les t\les T}\dbE\(\int_t^T |Z(t,s)|^2ds\)^{p\over 2}\nn\\
&\q\les K\dbE\[\sup_{t\in[0,T]}\Big|\int_t^T\3n g(t,s, y_s)ds\Big|^p\2n+\2n\sup_{t\in[0,T]}|\psi(t)|^p\].
\label{BSVIE-est}
\end{align}
\end{proposition}

\begin{proof}
By \autoref{Well-posed-ABSVIE}, BSVIE \rf{BSVIE-path-By} admits a unique solution $(Y(\cd),Z(\cd,\cd))$.
Thus, we only need to prove the pathwise continuity of $Y(\cd)$ and the estimate \rf{BSVIEest}.
Consider the following BSDE parameterized by $t\in[0,T]$ on $[0,T]$:
\bel{BSVIE-path-By-BSDE}
\eta(t,r)=\psi(t)+\int_t^T g(t,s,y_s)ds-\int_r^T\zeta(t,s)dW(s),\q r\in[0,T].
\ee
We emphasize that the initial value of the Lebesgue integral in the above is given by the parameter $t$ (rather than $r$).
By comparing \rf{BSVIE-path-By} and  \rf{BSVIE-path-By-BSDE},
it is obvious to see that
\bel{Proof-Prop:BSVIE-no-y-1}
Y(t)=\eta(t,t),\q Z(t,s)=\zeta(t,s),\q (t,s)\in\D[0,T].
\ee
Then by \autoref{BSDE-well-posed}, we get
\bel{Proof-Prop:BSVIE-no-y-2}
\sup_{t\in[0,T]}\dbE\(\int_t^T |Z(t,s)|^2ds\)^{p\over2}
\les K\dbE\[\sup_{t\in[0,T]}\Big|\int_t^T\3n g(t,s, y_s)ds\Big|^p\2n+\2n\sup_{t\in[0,T]}|\psi(t)|^p\].
\ee

\ms
For a fixed $1<p^\prime<p$, by \autoref{BSDE-well-posed} again we have
\begin{align}
|\eta(t,r)-\eta(t^\prime,r)|^{p^\prime}&\les K\dbE_r\[\Big(\int_{t^\prime}^T \big|g(t,s,y_s)-g(t^\prime,s,y_s)\big|ds\Big)^{p^\prime}+
\Big(\int^{t^\prime}_t|g(t,s,y_s)|ds\Big)^{p^\prime}+|\psi(t)-\psi(t^\prime)|^{p^\prime}\]\nn\\
&\les K\Big\{\dbE_r\big[|\xi|^{p^\prime}+\|y(\cd)\|^{p^\prime}_{\dbY}\big] \rho(|t-t^\prime|)^{p^\prime}+\dbE_r[|\xi|^{p^\prime}+\|y(\cd)\|^{p^\prime}_{\dbY}]|t-t^\prime|^{p^\prime}\Big\}\nn\\
&\les  K\Big\{\dbE_r\big[|\xi|^{p^\prime}+\|y(\cd)\|^{p^\prime}_{\dbY}\big] \Big\}\rho(|t-t^\prime|)^{p^\prime},
\qq 0\les t\les t^\prime\les T,\q r\in[0,T].\label{y-t-t-prime}
\end{align}
Note that $1<p^\prime<p$,  by Doob's maximum inequality, we get
\begin{align}
\dbE\[\sup_{r\in[0,T]}\dbE_r\big[|\xi|^{p^\prime}+\|y(\cd)\|^{p^\prime}_{\dbY}\big]\]
\les K\Big(\dbE\big[|\xi|^p+\|y(\cd)\|^p_{\dbY}\big]\Big)^{p^\prime\over p}<\infty. \label{Doob}
\end{align}
Let $\{t_i\}_{i\ges 1}$ be the rational numbers in $[0, T]$.
There exists an $\bar\Omega\subset\Omega$ such that $\dbP(\bar\Omega)=1$, $\eta(t_i,s,\omega)$ is continuous in $s$, and
\begin{align}
& \sup_{r\in[0,T]} |\eta(t_i,r) - \eta(t_j,r)|(\omega) \les  C_p(\omega)\rho(|t_i-t_j|),\q \forall (i, j),~ \forall \omega\in\bar \Omega,\nn\\
&\mbox{with}\q C_p(\omega)\deq \sup_{r\in[0,T]} \Big(K\dbE_r\big[|\xi|^{p^\prime}+\|y(\cd)\|^{p^\prime}_{\dbY}\big] \Big)^{1\over p^\prime}(\omega) <\infty,\q\forall \omega\in \bar\Omega,\label{wtYholder1}
\end{align}
where the inequality \rf{wtYholder1} is due to \rf{y-t-t-prime} and \rf{Doob}.
For any $t\in[0,T]$, by \rf{y-t-t-prime}, there exists an $\Omega^t\subset\Omega$ such that $\dbP(\Omega^t)=1$ and
\bea
\label{wtYholder2}
\sup_{r\in[0,T]} |\eta(t,r) - \eta(t_j,r)|(\omega) \les C_p(\omega)\rho(|t-t_j|),\q \forall j,~ \forall \omega\in \Omega^t\cap\bar\Omega.
\eea
Define
$$
\bar\eta(t,r,\omega)\deq \limsup_{t_j \to t} \eta(t_j,r,\omega),\q  \forall (t,r,\om)\in[0,T]^2\times\Omega.
$$
From \rf{wtYholder1}, note that for $\omega\in\bar\Omega$, the above $\limsup$ is actually a limit.
Then  for any $ \omega\in\bar\Omega$, $\bar\eta(t,r,\omega)$ is continuous in $r$, and
$$
 \sup_{r\in[0,T]} |\bar\eta(t,r) - \bar\eta(t^\prime,r)|(\omega) \les  C_p(\omega)\rho(|t-t^\prime|),\q \forall t, t^\prime\in[0,T].
$$
Thus,  for any $\omega\in \bar\Omega$, the function $\bar\eta(\cd,\cd,\om)$ is  continuous in $(t,r)\in [0,T]^2$.
Moreover, by \rf{wtYholder2} we have
$$
\eta(t,r,\omega)=\bar\eta(t,r,\omega),\q \forall r\in[0,T],~ \forall \omega\in\Omega^t\cap\bar\Omega.
$$
Since $\dbP(\Omega^t\cap \bar\Omega)=1$, so $\bar\eta(\cd,\cd)$ is a desired version of $\eta(\cd,\cd)$.
Then by always considering this version,  $\eta(\cd,\cd)$ is jointly continuous in $(t,r)$, a.s.
In particular, the process $Y(\cd)$, defined by $Y(t) = \eta(t,t);t\in[0,T]$, is continuous in $t$, a.s.

\ms

By \autoref{BSDE-well-posed}, we have
$$
 |\eta(r,r)|^{p^\prime}
\les K\dbE_r\[\Big|\int_r^T g(r,s,y_s)ds\Big|^{p^\prime}+|\psi(r)|^{p^\prime}\],
$$
which, together with $Y(r)=\eta(r,r)$, implies that
$$
 | Y(r)|^{p^\prime}
\les K\dbE_r\[\sup_{t\in[0,T]}\Big|\int_t^T g(t,s,y_s)ds\Big|^{p^\prime}+\sup_{t\in[0,T]}|\psi(t)|^{p^\prime}\].
$$
Note that both the processes  $Y(\cd)$ and
$
\dbE_\cd\big[\sup_{t\in[0,T]}\big|\int_t^T g(t,s,y_s)ds\big|^{p^\prime}+\sup_{t\in[0,T]}|\psi(t)|^{p^\prime}\big]
$
are pathwise continuous.
Then by  Doob's maximum inequality again (noting $p>p^\prime$), we get
$$
\dbE\[\sup_{t\in[0,T]} | Y(t)|^p\]
\les K\dbE\[\sup_{t\in[0,T]}\Big|\int_t^T g(t,s,y_s)ds\Big|^p+\sup_{t\in[0,T]}|\psi(t)|^p\].
$$

\end{proof}


\ms\no
{\it\textbf{Proof of \autoref{thm:well-posedness-A}}.}
By \autoref{Prop:BSVIE-no-y}, the map
$$
\G(y(\cd),z(\cd,\cd))\deq (Y(\cd),Z(\cd,\cd)),\q \forall(y(\cd),z(\cd,\cd))\in L_{\dbF}^p(\Om;C([0,T];\dbR^n))\times L^p_{\dbF}(\Om;L^2(\D[0,T];\dbR^n)),
$$
is well-defined on the space $L_{\dbF}^p(\Om;C([0,T];\dbR^n))\times L^p_{\dbF}(\Om;L^2(\D[0,T];\dbR^n))$,
where $(Y(\cd),Z(\cd,\cd))$ is the unique solution to \rf{BSVIE-path-By}.
Let $(\ti y(\cd),\ti z(\cd,\cd))\in L_{\dbF}^p(\Om;C([0,T];\dbR^n))\times L^p_{\dbF}(\Om;L^2(\D[0,T];\dbR^n))$
and denote $(\ti Y(\cd),\ti Z(\cd,\cd))=\G(\ti y(\cd),\ti z(\cd,\cd))$.
Note that for any $\d\in[0,T)$, by \autoref{Prop:BSVIE-no-y},
\begin{align}
&\dbE\Big[\sup_{T-\d\les t\les T} |Y(t)-\ti Y(t)|^p \Big] + \sup_{T-\d\les t\les T}\dbE\(\int_t^T |Z(t,s)-\ti Z(t,s)|^2ds\)^{p\over 2}\nn\\
&\q\les K\d^p \dbE\Big[\sup_{T-\d\les t\les T} |y(t)-\ti y(t)|^p \Big].
\end{align}
Then by the fixed point theorem, we can get that BSVIE \rf{BSVIE-path} admits a unique solution $(\bar Y(\cd),\bar Z(\cd,\cd))\in L_{\dbF}^2(\Om;C([T-\d,T];\dbR^n))\times L^p_{\dbF}(\Om;L^2(\D[T-\d,T];\dbR^n))$ for some small enough $\d>0$.

\ms
Next, we consider the following BSVIE
\bel{BSVIE-2d}
Y(t)=\psi(t)+\int_t^T \[g(t,s,Y_s)\textbf{1}_{[0,T-\d)}(s)+ g(t,s,\bar Y_s)\textbf{1}_{[T-\d,T]}(s) \]ds-\int_t^T Z(t,s)dW(s),\q t\in[0,T].
\ee
The map
$$
\h\G(y(\cd),z(\cd,\cd))\deq (Y(\cd),Z(\cd,\cd)),\q \forall(y(\cd),z(\cd,\cd))\in L_{\dbF}^p(\Om;C([0,T];\dbR^n))\times L^p_{\dbF}(\Om;L^2(\D[0,T];\dbR^n)),
$$
is still well-defined on the space $L_{\dbF}^p(\Om;C([0,T];\dbR^n))\times L^p_{\dbF}(\Om;L^2(\D[0,T];\dbR^n))$,
where $(Y(\cd),Z(\cd,\cd))$ is the unique solution to the following BSVIE:
$$
Y(t)=\psi(t)+\int_t^T \[g(t,s,y_s)\textbf{1}_{[0,T-\d)}(s)+ g(t,s,\bar Y_s)\textbf{1}_{[T-\d,T]}(s) \]ds-\int_t^T Z(t,s)dW(s),\q t\in[0,T].
$$
We emphasize that in the above, the generator is independent of $y(s),s\in[T-\d,T]$.
Let $(\ti y(\cd),\ti z(\cd,\cd))\in L_{\dbF}^p(\Om;C([0,T];\dbR^n))\times L^p_{\dbF}(\Om;L^2(\D[0,T];\dbR^n))$
and denote $(\ti Y(\cd),\ti Z(\cd,\cd))=\h\G(\ti y(\cd),\ti z(\cd,\cd))$.
Then we have
\begin{align*}
&\dbE\Big[\sup_{T-2\d\les t\les T} |Y(t)-\ti Y(t)|^p \Big]
+ \sup_{T-2\d\les t\les T}\dbE\(\int_t^T |Z(t,s)-\ti Z(t,s)|^2ds\)^{p\over 2}\\
&\q\les K\d^p \dbE\Big[\sup_{T-2\d\les t\les T-\d} |y(t)-\ti y(t)|^p \Big].
\end{align*}
By the fixed point theorem again, we can get that BSVIE \rf{BSVIE-2d} admits a unique solution $(\h Y(\cd),\h Z(\cd,\cd))\in L_{\dbF}^p(\Om;C([T-2\d,T];\dbR^n))\times L^p_{\dbF}(\Om;L^2(\D[T-2\d,T];\dbR^n))$.
Clearly,
$$
\h Y(t)=\bar Y(t),\q \h Z(t,s)=\bar Z(t,s),\q T-\d\les t\les s\les T.
$$
Thus, BSVIE \rf{BSVIE-path} admits a unique solution on $[T-2\d,T]$.
By continuing such a procedure, we obtain the existence and uniqueness of a solution to BSVIE \rf{BSVIE-path}.
Moreover, by \autoref{Prop:BSVIE-no-y}, we have
\begin{align*}
&\dbE\Big[\sup_{0\les t\les T} |Y(t)|^p \Big] + \sup_{0\les t\les T}\dbE\(\int_t^T |Z(t,s)|^2ds\)^{p\over 2}\\
&\q\les K\dbE\int_{0}^T \|Y_s\|^p_{\dbY_s}ds+K\dbE\[\sup_{t\in[0,T]}\Big|\int_t^T g(t,s,y_s)ds\Big|^p+\sup_{t\in[0,T]}|\psi(t)|^p\].
\end{align*}
Then by Gr\"{o}nwall's inequality, we get the estimate \rf{BSVIEest} immediately.

\subsection{Path-Dependent Type-I BSVIE with Adaptedness Conditions}

Consider the following path-dependent BSVIE with progressively measurable generators:

\bel{BSVIE-path-2}
Y(t)=\psi(t)+\int_t^T g(t,s,Y_s,Z(t,s))ds-\int_t^TZ(t,s)dW(s),\q t\in[0,T].
\ee

\ss
{\bf(H4)$_p^\prime$}
The map $\psi:[0,T]\to\dbR^n$ is $\cF_T$-measurable and the map $g:\D[0,T]\times \dbY\times\dbR^n\times\Om\to\dbR^n$
is $\cB(\D[0,T]\times \dbY\times\dbR^n)\otimes\cF_T$-measurable.
For any $(t,s)\in\D[0,T]$, the geneator $g(t,s,\cd,\cd)$ satisfies:
\begin{align}
&g(t,s,{\bf y},z)=g(t,s,{\bf y_s},z),\q \forall ({\bf y},z)\in\dbY\times\dbR^n,\nn\\
&g(t,s,y_s,z)\in\cF_s,\q \forall ( y(\cd),z)\in L^p_{\dbF}(\Om; C([0,T];\dbR^n))\times\dbR^n.\label{adapt-condition}
\end{align}
There exist a constants $L>0$ and a modulus of continuous function $\rho(\cd)$ and a positive random variable $\xi\in L^p_{\cF_T}(\Om;\dbR_+)$ such that:
\begin{align}
&|g(t,s,\By_s,z)-g( t,s,\By^\prime_s,z^\prime)|\les L\big[\|\By_s-\By^\prime_s\|_{\dbY_s}+|z-z^\prime|\big],\nn\\
&|g(t-\d,s,\By_s,z)-g( t,s,\By_s,z)|+|\psi(t-\d)-\psi(t)|\les (|\xi|+\|\By\|_{\dbY_s}+|z|)\rho(\d),\nn\\
&\qq\qq\qq\qq\qq\q \forall (t,s)\in\D[0,T],\, \By_s,\By_s^\prime\in\dbY_s,\,(z,z^\prime)\in\dbR^n,\,\d\in[0,t].
\end{align}
Moreover,
\bel{IIp}\dbE\[\sup_{t\in[0,T]}\Big|\int_t^T g(t,s,{\bf 0},0)ds\Big|^p+\sup_{t\in[0,T]}|\psi(t)|^p\]<\infty.
\ee

\begin{remark}\rm
Note that the adaptedness condition \rf{adapt-condition} is much stronger than the following
progressively measurable condition:
$$
g(t,s,{\bf y_s},z)\in\cF_s,\q \forall ({\bf y},z)\in\dbY\times\dbR^n.
$$
\end{remark}
Combining the arguments employed in \autoref{Prop:BSVIE-no-y} and \cite[Theorem 3.7]{Yong2008} together,
the following result can be obtained easily. We omit it here.

\begin{theorem}\label{thm:well-posedness-A-P}
Let {\rm(H4)$_p^\prime$} hold with $p>1$.
Then  BSVIE \rf{BSVIE-path-2} admits a unique adapted solution $(Y(\cd),Z(\cd,\cd))$
such that $Y(\cd)$ is continuous in $t$ and the following estimate holds true:
\begin{align}
&\dbE\Big[\sup_{0\les t\les T} |Y(t)|^p \Big] + \sup_{0\les t\les T}\dbE\(\int_t^T |Z(t,s)|^2ds\)^{p\over2}\nn\\
&\q\les K\dbE\[\sup_{t\in[0,T]}\Big|\int_t^T g(t,s,{\bf 0},0)ds\Big|^p+\sup_{t\in[0,T]}|\psi(t)|^p\].
\end{align}
\end{theorem}

\section{Connections with  FBSDEs}\label{Sec:BSVIE-FBSDE}

In this section, we shall show that if the generator of BSVIEs is allowed to be anticipating,
then it has an interesting connection with (coupled) FBSDEs.

\ms

Let us consider the following coupled FBSDE parameterized by $t\in[0,T)$:
\bel{FBSDE}\left\{
\begin{aligned}
X^{t}(s)&=x^{t}+\int_{t}^s b^t(r,Z^{t}(r))dr,\\
Y^{t}(s)&=\xi^t X^{t}(T)+\int_{s}^T g^t(r,Z^{t}(r))dr-\int_s^T Z^{t}(r)dW(r),
\end{aligned}
\right.\q s\in[t,T].
\ee
In the above, for any fixed $t\in[0,T)$, $b^t,\,g^t:[t,T]\times\dbR^n\times\Om\to\dbR^n$ are
$\dbF$-progressively measurable and Lipschitz continuous with respect to $z$,
$x^t$ is an $\cF_t$-measurable and square-integrable random variable,
and $\xi^t$  is an $\cF_T$-measurable and essentially bounded random variable.

\begin{proposition}\label{Prop:FBSDE-BSVIE}
\footnote{Hanxiao Wang would like to thank Chenchen Mou (City University of Hong Kong) for  suggesting him to study the connection between BSVIEs
with anticipating generators and coupled FBSDEs.} \sl
The well-posedness of FBSDE \rf{FBSDE} on any time horizon $[t,T],\,t\in[0,T]$ is equivalent to the well-posedness of the following BSVIE:
\begin{align}
Y(t)&=\xi^t x^t +\int_t^T\big[ \xi^t b^t(r,Z(t,r))+ g^t(r,Z(t,r))\big]dr-\int_t^T Z(t,r)dW(r),\qq t\in[0,T].\label{FBSDE-BSVIE}
\end{align}
\end{proposition}

\begin{proof}
The proof is divided into two steps.

\ms

{\bf Step 1.} Suppose that BSVIE \rf{FBSDE-BSVIE} has a solution $(Y(\cd),Z(\cd,\cd))$.
For any fixed $t$, we consider the following BSDE with unknown processes $(\cY(t,\cd),\cZ(t,\cd))$:
\bel{BSDE-parameter}
\cY(t,s)=\xi^t x^t +\int_t^T\big[ \xi^t b^t(r,Z(t,r))+ g^t(r,Z(t,r))\big]dr-\int_s^T \cZ(t,r)dW(r),\qq s\in[t,T].
\ee
Taking $\xi^t x^t +\int_t^T\big[ \xi^t b^t(r,Z(t,r))+ g^t(r,Z(t,r))\big]dr$ as a terminal term, by the standard results of BSDEs,
we get that BSDE \rf{BSDE-parameter} admits a unique solution $(\cY(t,\cd),\cZ(t,\cd))$.
Letting $s=t$ and $\cY(t)=\cY(t,t)$ in \rf{BSDE-parameter}, we get
\begin{align}
\cY(t)&=\xi^t x_t +\int_t^T\big[ \xi^t b^t(r,Z(t,r))+ g^t(r,Z(t,r))\big]dr-\int_t^T \cZ(t,r)dW(r),\qq t\in[0,T].\label{BSDE-BSVIE1}
\end{align}
Then by the uniqueness of the solution to BSVIE \rf{BSDE-BSVIE1} (with $(\cY(\cd),\cZ(\cd,\cd))$ being unknown processes),
we have
$$
Y(t)=\cY(t), \q Z(t,s)=\cZ(t,s),\qq (t,s)\in\D[0,T].
$$
Substituting the above into \rf{BSDE-parameter} yields that
$$
\cY(t,s)=\xi^t x^t +\int_t^T\big[ \xi^t b^t(r,\cZ(t,r))+ g^t(r,\cZ(t,r))\big]dr-\int_s^T \cZ(t,r)dW(r),\qq s\in[t,T].
$$
Let
$$
X^t(s)= x^t +\int_t^s  b^t(r,\cZ(t,r))dr,\q Y^t(s)=\cY(t,s), \qq Z^t(s)=\cZ(t,s),\qq (t,s)\in\D[0,T].
$$
Then it is easily checked that for any $t\in[0,T)$,
$(X^t(\cd),Y^t(\cd),Z^t(\cd))$ is a solution of FBSDE \rf{FBSDE} on $[t,T]$.

\ms

{\bf Step 2.} Suppose that for any $t\in[0,T)$, FBSDE \rf{FBSDE} has a solution $(X^t(\cd),Y^t(\cd),Z^t(\cd))$.
Then from the fact
$$
X^{t}(T)=x^t+\int_{t}^T b^t(r,Z^{t}(r))dr,
$$
we have
$$
Y^t(s)=\xi^t x^t +\int_t^T \xi^t b^t(r,Z^t(r))dr +\int_s^T g^t(r,Z^t(r))dr-\int_s^T Z^t(r)dW(r),\qq s\in[t,T].
$$
By taking $s=t$ in the above, we get
$$
Y^t(t)=\xi^t x^t +\int_t^T \xi^t b^t(r,Z^t(r))dr +\int_t^T g^t(r,Z^t(r))dr-\int_t^T Z^t(r)dW(r),\qq t\in[0,T].
$$
Thus, the process $(Y(\cd),Z(\cd,\cd))$, defined by
$$
Y(t)\deq Y^t(t),\q Z(t,s)\deq Z^t(s),\qq (t,s)\in\D[0,T],
$$
is a solution to BSVIE \rf{FBSDE-BSVIE}.
\end{proof}

\begin{remark}\label{remark-fbsde}\rm
Note that  \rf{FBSDE} is  a multi-dimensional coupled  FBSDE and the forms of  $b^t(\cd), g^t(\cd)$ are very general.
Thus, the coefficients $\xi^t$, $b^t(\cd)$ and $g^t(\cd)$ might not satisfy the monotone condition imposed for
coupled FBSDEs in the literature; see, \cite{Hu-Peng1995,Yong1997,Ma-Yong1999,Zhang2017}, for example.
To ensure that \rf{FBSDE} admits a unique solution for any given time horizon $[t,T]\subseteq[0,T]$,
we can impose the small-time assumption, that is, $\|\xi^t b^t_z(\cd)\|_{L^\infty}$ is controlled by some constant $K_T>0$ depending on $T$.
On the other hand, by {\rm\autoref{Prop:FBSDE-BSVIE}}, we know that the well-posedness of FBSDE \rf{FBSDE}
is equivalent to that of the BSVIE \rf{FBSDE-BSVIE} with the following generator:
$$
g(t,s,z)\equiv \xi^t b^t(s,z)+ g^t(s,z),\q (t,s,z)\in\D[0,T]\times\dbR^n.
$$
Clearly, $g(\cd)$ is anticipating, because $\xi^t$ is $\cF_T$-measurable.
The corresponding decomposition (see \rf{decomposition}) is given by
$$
 g_1(t,s,y,z,\h z)=\dbE_s[\xi^t]b^t(s,z)+g^t(s,z),\q  g_0(t,s,y,z,\h z)=\{\xi^t-\dbE_s[\xi^t]\}b^t(s,z).
$$
The corresponding $L_z^0(\cd,\cd)$ can be given by $\|\{\xi^t-\dbE_s[\xi^t]\} b^t_z(\cd)\|_{L^\infty}$.
Thus, \rf{K<1} is in some sense equivalent to the small-time assumption for coupled FBSDEs.
\end{remark}

\section{Conclusion}\label{Sec:Concludng}

In conclusion,  the BSVIEs with anticipating generators are explored for the first time.
The path-dependent BSVIEs are also introduced and studied.
Under proper conditions, the well-posedness of these kinds of BSVIEs are established.
Some essential differences between BSDEs and BSVIEs are revealed: (i) If the generator is allowed to be anticipating,
BSDEs are ill-posed in general, but they can still have unique adapted solutions within the framework of BSVIEs;
(ii) For anticipated BSDEs, the adaptedness condition is not avoidable, but it is no longer necessary within the framework of
(path-dependent) BSVIEs. In addition, we find that if the generator of BSVIEs is allowed to be anticipating,
it has an interesting connection with coupled FBSDEs.

\ms
The following provides a clear path of how  a BSVIE is reduced to a BSDE by changing the measurable condition of
the generator.

\ms
\textbf{(I)} Consider the following BSVIE with the anticipating generator $g(\cd)$:
\bel{Con-1}
Y(t)=\psi(t)+\int_t^T g(t,s,Y(s),Z(t,s))ds-\int_t^TZ(t,s)dW(s),\q t\in[0,T].
\ee
If $g(t,s,y,z)$ and $\psi(t)$ are independent of $t$, \rf{Con-1} is still a BSVIE, with the form:
\bel{Con-2}
Y(t)=\xi+\int_t^T g(s,Y(s),Z(t,s))ds-\int_t^TZ(t,s)dW(s),\q t\in[0,T].
\ee
If the generator $g(\cd)$ is adapted, \rf{Con-2} is reduced to a BSDE:
\bel{}
Y(t)=\xi+\int_t^T g(s,Y(s),Z(s))ds-\int_t^TZ(s)dW(s),\q t\in[0,T].
\ee

\ms
\textbf{(II)} Consider the following path-dependent BSVIE with the anticipating generator $g(\cd)$:
\bel{Con-4}
Y(t)=\psi(t)+\int_t^T g(t,s,Y_s)ds-\int_t^TZ(t,s)dW(s),\q t\in[0,T].
\ee
If $g(t,s,\By)$ and $\psi(t)$ are independent of $t$, \rf{Con-4} is still a BSVIE, with the form:
\bel{Con-5}
Y(t)=\xi+\int_t^T g(s,Y_s)ds-\int_t^TZ(t,s)dW(s),\q t\in[0,T].
\ee
If the generator $g(s,y_s)$ is $\cF_T$-measurable for any $y(\cd)\in L_{\dbF}^p(\Om; C([0,T];\dbR^n))$,
in which case  $g(s,\By_s)$ can be $\cF_s$-measurable for any $\By(\cd)\in C([0,T];\dbR^n)$,
 \rf{Con-5} is still a BSVIE, with the form:
\bel{Con-6}
Y(t)=\xi+\int_t^T g(s,Y_s)ds-\int_t^TZ(t,s)dW(s),\q t\in[0,T].
\ee
If the generator $g(s,y_s)$ is $\cF_s$-measurable for any $y(\cd)\in L_{\dbF}^p(\Om; C([0,T];\dbR^n))$,
\rf{Con-6} is reduced to an anticipated BSDE:
\bel{}
Y(t)=\xi+\int_t^T g(s,Y_s)ds-\int_t^TZ(s)dW(s),\q t\in[0,T].
\ee

\section{Appendix: BSDEs with Progressively Measurable Generators}

First, let us recall the martingale moment inequalities (see \cite[page 163]{Karatzas-Shreve1988}).

\begin{proposition}\label{Martingale Moment}
For any $p>1$, there exists a constant $K_p>0$ such that for any $z(\cd)\in L^p_\dbF(\Om;$ $L^2(0,T;\dbR^n))$,
\bel{inequality}\dbE\(\int_0^T|z(s)|^2ds\)^{p\over2}
\les K_p\dbE\Big|\int_0^Tz(s)dW(s)\Big|^p.\ee
\end{proposition}

Note that if $p=2$, then $K_2=1$ and the equality holds in the above.

\ms

Let us now consider the following BSDE:
\bel{BSDE0}Y(t)=\xi+\int_t^Tg_1(s,Z(s))ds-\int_t^TZ(s)dW(s),\qq t\in[0,T].\ee
We introduce the following hypothesis.

\ms

{\bf(H0)$_p$} Let $g_1:[0,T]\times\dbR^n\times\Om\to\dbR^n$ satisfy the following: For each $z\in\dbR^n$, $t\mapsto g_1(t,z)$ is $\dbF$-progressively measurable. There exists a function $L_z^1(\cd)\in L^2(0,T)$ such that
$$|g_1(t,z_1)-g_1(t,z_2)|\les L_z^1(t)|z_1-z_2|,\qq\forall t\in[0,T],~z\in\dbR^n.$$
Further, $g_1(\cd\,,0)\in L_\dbF^p(\Om;L^1(0,T;\dbR^n))$.

\ms

In what follows, we will denote
$$L_0^1(s)=|g_1(s,0)|,\qq s\in[0,T].$$

\ms

\bp{BSDE-prop1} \sl Let {\rm(H0)$_p$} hold. Then for any $\xi\in L^p_{\cF_T}(\Om;\dbR^n)$, BSDE \rf{BSDE0} admits a unique adapted solution $(Y(\cd),Z(\cd))$, and the following holds:
\bel{estimate1}\[\sup_{t\in[0,T]}\dbE|Y(t)|^p\]^{1\over p}+\[\dbE\(\int_0^T|Z(s)|^2ds\)^{p\over2}\]^{1\over p}
\les\h K_p\(\dbE|\xi|^p\)^{1\over p}+N_p\h K_p\[\dbE\(\int_0^TL_0^1(s)ds\)^p\]^{1\over p},\ee
where
\bel{K_p^0-1}\begin{aligned}
\h K_p&=\min\Big\{\[1+2\ti K_pN \({\sqrt N\over\sqrt N-\sqrt{\bar K_p}}\)\]\({\sqrt N\over\sqrt N-\sqrt{\bar K_p}}\)^N\bigm|N\in\dbN,~N>\bar K_p\Big\}\\
&\equiv\[1+2\ti K_pN_p\({\sqrt{N_p}\over\sqrt{N_p}-\sqrt{\bar K_p}}\)\]\({\sqrt{N_p}\over\sqrt{N_p}-\sqrt{\bar K_p}}\)^{N_p},
\end{aligned}
\ee
with
\bel{bar-K} \ti K_p =K_p^{1\over p},\qq \bar K=4\ti K^2_p\int_0^TL_z^1(s)^2ds.\ee
Further, if $g_i(\cd\,,\cd)$ satisfies {\rm(H0)$_p$}, $\xi_i\in L^p_{\cF_T}(\Om;\dbR^n)$, and $(Y_i(\cd),Z_i(\cd))$
is the adapted solution to the corresponding BSDE, $i=1,2$, then
\bel{estimate2}\ba{ll}
\ds\[\sup_{t\in[0,T]}\dbE|Y_1(t)-Y_2(t)|^p\]^{1\over p}+\[\dbE\(\int_0^T|Z_1(s)-Z_2(s)|^2ds\)^{p\over2}\]^{1\over p}\\
\ns\ds\q\les\h K_p\(\dbE|\xi_1-\xi_2|^p\)^{1\over p}+N_p\h K_p\[\dbE\(\int_0^T|g_1(s,Z_1(s))-g_2(s,Z_1(s))|ds\)^p\]^{1\over p}.
\ea\ee

\ep

\it Proof. \rm First of all, BSDE \rf{BSDE0} leads to
$$Y(t)=\dbE_t\[\xi+\int_t^Tg_1(s,Z_1(s))ds\].$$
Thus,
$$\ba{ll}
\ns\ds\(\dbE|Y(t)|^p\)^{1\over p}\les\(\dbE|\xi|^p\)^{1\over p}
+\[\dbE\(\int_t^TL^1_0(s)ds\)^p\]^{1\over p}+\[\dbE\(\int_t^TL_z^1(s)|Z(s)|ds\)^p\]^{1\over p}\\
\ns\ds\q\les\(\dbE|\xi|^p\)^{1\over p}+\[\dbE\(\int_t^TL^1_0(s)ds\)^p\]^{1\over p}
+\(\int_t^TL_z^1(s)^2ds\)^{1\over2}\[\dbE\(\int_t^T|Z(s)|^2ds\)^{p\over2}
\]^{1\over p}.\ea$$
As a result,
$$\sup_{s\in[t,T]}\(\dbE|Y(s)|^p\)^{1\over p}\les\(\dbE|\xi|^p\)^{1\over p}
+\[\dbE\(\int_t^TL^1_0(s)ds\)^p\]^{1\over p}+\(\int_t^TL_z^1(s)^2ds\)^{1\over2}\[\dbE\(\int_t^T|Z(s)|^2ds\)^{p\over2}
\]^{1\over p}.$$
Also, making use of the martingale moment inequality (\autoref{Martingale Moment}), one has
\bel{|Z|1}\ba{ll}
\ns\ds\[\dbE\(\int_t^T|Z(s)|^2ds\)^{p\over2}\]^{1\over p}\les \ti K_p\[\dbE\Big|
\int_t^TZ(s)dW(s)\Big|^p\]^{1\over p}\\
\ns\ds\q\les \ti K_p\[\dbE\Big|\xi+\int_t^Tg_1(s,Z(s))ds
-Y(t)\Big|^p\]^{1\over p}\\
\ns\ds\q\les 2 \ti K_p\Big\{\(\dbE|\xi|^p\)^{1\over p}+\[\dbE\(\int_t^TL^1_0(s)ds\)^p\]^{1\over p}+\(\int_t^TL^1_z(s)^2ds
\)^{1\over2}\[\dbE\(\int_t^T|Z(s)|^2ds\)^{p\over2}\]^{1\over p}\Big\},\ea\ee
where $\ti K_p =K_p^{1\over p}$ with $K_p$ being given by \autoref{Martingale Moment}.
For any $0\les t\les\t\les T$, denote
$$\ba{ll}
\ns\ds\Th_y(t,\t)=\(\sup_{s\in[t,\t]}\dbE|Y(s)|^p\)^{1\over p},\qq\Th_z(t,\t)=\[\dbE\(\int_t^\t|Z(s)|^2ds\)^{p\over2}\]^{1\over p},\\
\ns\ds\cL(t,\t)=\(\int_t^\t L^1_z(s)^2ds\)^{1\over2},\qq\cL_0(t,\t)=\[\dbE\(\int_t^\t L_0^1(s)ds\)^p\]^{1\over p}.\ea$$
Then the above procedure leads to
$$\ba{ll}
\ns\ds\Th_y(t,\t)\les\Th_y(\t,\t)+\cL_0(t,\t)+\cL(t,\t)\Th_z(t,\t),\\
\ns\ds\Th_z(t,\t)\les2 \ti K_p\[\Th_y(\t,\t)+\cL_0(t,\t)+\cL(t,\t)\Th_z(t,\t)\].\ea$$
Thus, by letting $\t-t>0$ small so that $2\ti K_p\cL(t,\t)<1$, one has
$$\ba{ll}
\ns\ds\Th_z(t,\t)\les{2\ti K_p\[\Th_y(\t,\t)+\cL_0(t,\t)\]\over1-2\ti K_p\cL(t,\t)},\qq
\Th_y(t,\t)\les{\Th_y(\t,\t)+\cL_0(t,\t)\over1-2\ti K_p\cL(t,\t)}.\ea$$
%
%
%
Pick $0=t_0<t_1<t_2<\cds<t_N=T$ so that
\bel{alpha}\a\equiv{1\over1-2\ti K_p\cL(t_{k-1},t_k)}\equiv{1\over\ds1-2\ti K_p\(\int_{t_{k-1}}^{t_k}
L^1_z(s)^2ds\)^{1\over2}}\in(1,\infty),\q k=1,2,\cds,N.\ee
Then we have
\begin{align*}
\Th_y(t_{k-1},t_k)&\les\a\Th_y(t_k,t_k)+\a\cL_0(t_{k-1},t_k)\les\a\Th_y(t_k,t_{k+1})+\a\cL_0(t_{k-1},t_k)\\
&\les\a^2\Th_y(t_{k+1},t_{k+2})+\a^2\cL_0
(t_k,t_{k+1})+\a\cL_0(t_{k-1},t_k)\les\cds\\
&\les\a^{N-k+1}\Th_y(T,T)+\sum_{i=k}^N\a^{i-k+1}\cL_0(t_{i-1},t_i)\les \a^N\Th_y(T,T)+N\a^N\cL_0(0,T),
\end{align*}
and
$$\ba{ll}
\ns\ds\Th_z(0,T)\les\sum_{k=1}^N\Th_z(t_{k-1},t_k)\les\sum_{k=1}^N2 \ti K_p\a\[
\Th_y(t_k,t_k)+\cL_0(t_{k-1},t_k)\]\\
\ns\ds\qq\qq\les2\ti K_pN\a^{N+1}\Th_y(T,T)+2\ti K_pN^2\a^{N+1}\cL_0(0,T).\ea$$
That is
$$\ba{ll}
\ns\ds\[\sup_{t\in[0,T]}\dbE|Y(t)|^p\]^{1\over p}\les\a^N\(\dbE|\xi|^p\)^{1\over p}+N\a^N\[\dbE\(\int_0^TL_0(s)ds\)^p\]^{1\over p},\\
\ns\ds\[\dbE\(\int_0^T|Z(s)|^2ds\)^{p\over2}\]^{1\over p}
\les2\ti K_pN\a^{N+1}\(\dbE|\xi|^p\)^{1\over p}+2\ti K_pN^2\a^{N+1}\[\dbE\(\int_0^TL_0(s)ds\)^p\]^{1\over p}.\ea$$
Hence,
$$\ba{ll}
\ns\ds\[\sup_{t\in[0,T]}\dbE|Y(t)|^p\]^{1\over p}+\[\dbE\(\int_0^T|Z(s)|^2ds\)^{p\over2}\]^{1\over p}\\
\ns\ds\q\les(1+2\ti K_pN\a)\a^N\(\dbE|\xi|^p\)^{1\over p}+(1+2\ti K_pN\a)N\a^N\[\dbE\(\int_0^TL_0(s)ds\)^p\]^{1\over p}.\ea$$

\ms
We now determine $N$ and $\a$ in terms of the known conditions. From \rf{alpha}, we have have the relation between $\a$ and $N$:
$$N{(\a-1)^2\over\a^2}=4\ti K^2_p\int_0^TL_z^1(s)^2ds\equiv\bar K_p.$$
Thus, for given large $N\in\dbN$, we have
$$0=N(\a-1)^2-\bar K_p\a^2=(N-\bar K_p)\a^2-2N\a+N.$$
This means that for any integer $N>\bar K$,
$$1<\a={2N+\sqrt{4N^2-4(N-\bar K_p)N}\over2(N-\bar K_p)}={N+\sqrt{N\bar K_p}\over N-\bar K_p}={\sqrt N\over\sqrt N-\sqrt{\bar K_p}}.$$
Let
$$f(x)=\[1+2\ti K_px\({\sqrt x\over\sqrt x-\sqrt{\bar K_p}}\)\]\({\sqrt x\over\sqrt x-\sqrt{\bar K_p}}\)^x,\qq x>\bar K_p.$$
Clearly,
$$\lim_{x\to\bar K_p}f(x)=\lim_{x\to\infty}f(x)=\infty.$$
Hence, the following is well-defined:
\bel{K_p^0}\begin{aligned}
\h K_p&=\min\Big\{\[1+2\ti K_pN \({\sqrt N\over\sqrt N-\sqrt{\bar K_p}}\)\]\({\sqrt N\over\sqrt N-\sqrt{\bar K_p}}\)^N\bigm|N\in\dbN,~N>\bar K_p\Big\}\\
&\equiv\[1+2\ti K_pN_p\({\sqrt{N_p}\over\sqrt{N_p}-\sqrt{\bar K_p}}\)\]\({\sqrt{N_p}\over\sqrt{N_p}-\sqrt{\bar K_p}}\)^{N_p}.
\end{aligned}
\ee
Consequently, we have
\bel{estimate4}\[\sup_{t\in[0,T]}\dbE|Y(t)|^p\]^{1\over p}+\[\dbE\(\int_0^T|Z(s)|^2ds\)^{p\over2}\]^{1\over p}
\les\h K_p\(\dbE|\xi|^p\)^{1\over p}+N_p\h K_p\[\dbE\(\int_0^TL_0^1(s)ds\)^p\]^{1\over p}.\ee

\ms

Now, let $g_i(\cd\,,\cd)$ satisfy (H0)$_p$, $\xi_i\in L^p_{\cF_T}(\Om;\dbR^n)$ and let $(Y_i(\cd),Z_i(\cd))$ be the corresponding adapted solution:
$$Y_i(t)=\xi_i+\int_t^Tg_i(s,Z_i(s))ds-\int_t^TZ_i(s)dW(s),\qq t\in[0,T],~i=1,2.$$
Then we have the following BSDE for $(Y_1(\cd)-Y_2(\cd),Z_1(\cd)-Z_2(\cd))$:
\begin{align*}
 Y_1(t)-Y_2(t)&=\xi_1-\xi_2+\int_t^T\big\{g_1(s,Z_1(s))-g_2(s,Z_1(s))+G(s)[Z_1(s)-Z_2(s)]\big\}ds\\
&\q-\int_t^T[Z_1(s)-Z_2(s)]dW(s),\end{align*}
where
$$G(s)={[g_2(s,Z_1(s))-g_2(s,Z_2(s))][Z_1(s)-Z_2(s)]^\top\over|Z_1(s)-Z_2(s)|^2}
{\bf1}_{\{Z_1(s)\ne Z_2(s)\}}.$$
Clearly,
$$|G(s)|\les L_z^1(s),\qq s\in[0,T],~\as$$
By \rf{estimate4}, we obtain
$$\ba{ll}
\ds\[\sup_{t\in[0,T]}\dbE|Y_1(t)-Y_2(t)|^p\]^{1\over p}+\[\dbE\(\int_0^T|Z_1(s)-Z_2(s)|^2ds\)^{p\over2}\]^{1\over p}\\
\ns\ds\q\les\h K_p\(\dbE|\xi_1-\xi_2|^p\)^{1\over p}+N_p\h K_p\[\dbE\(\int_0^T|g_1(s,Z_1(s))-g_2(s,Z_1(s))|ds\)^p\]^{1\over p}.
\ea$$
This completes the proof. \endpf

\ms

We now look at the following Type-I BSVIE:
\bel{BSVIE-I0}Y(t)=\psi(t)+\int_t^Tg_1(t,s,Z(t,s))ds-\int_t^TZ(t,s)dW(s),\qq t\in[0,T].\ee
We introduce the following assumption.

\ms

{\bf(H0)$_p'$} Let $g_1:\D[0.T]\times\dbR^n\times\Om\to\dbR^n$ satisfy the following: For each $(t,z)\in[0,T]\times\dbR^n$, $s\mapsto g_1(t,s,z)$ is $\dbF$-progressively measurable on $[t,T]$.
There exists a deterministic function $L_z^1(\cd\,,\cd)\in L^2(\D[0,T];\dbR_+)$ such that
$$|g_1(t,s,z_1)-g_1(t,s,z_2)|\les L_z^1(t,s)|z_1-z_2|,\qq\forall(t,s)\in\D,~z\in\dbR^n.$$
with the property
\bel{K<1-2}K^0_p\sup_{t\in[0,T]}\(\int_t^TL^1_z(t,s)^2ds\)^{p\over2}<1,\ee
where $K_p^0\equiv(\h K_p)^p$ and $\h K_p$ is given by \rf{bar-K} with $\bar K$ replaced by
\bel{}
\bar K=4\ti K^2_p\sup_{t\in[0,T]}\int_t^TL_z^1(t,s)^2ds.
\ee
Further, $g_1(\cd\,,\cd\,,0)\in L_\dbF^p(\Om;L^1(\D[0,T];\dbR^n))$.

\ms

%
%
The following proposition is concerned with the above BSVIE.

\begin{proposition}\label{simple BSVIE} \sl Let {\rm(H0)$_p'$} hold. Then for any $\psi(\cd)\in L^p_{\cF_T}(0,T;\dbR^n)$, BSVIE \rf{BSVIE-I0} admits a unique adapted solution $(Y(\cd),Z(\cd\,,\cd))\in\cH_\D^p[0,T]$ and with $K_p^0=(\h K_p)^p$, the following holds:
\bel{9.13}\dbE|Y(t)|^p+\dbE\(\int_t^T|Z(t,s)|^2ds\)^{
p\over2}\les K_p^0\dbE\Big|\psi(t)+\int_t^Tg(t,s,0)ds\Big|^p,\q\ae\;t\in[0,T].\ee
%
%
%
%
If $g_i(\cd\,,\cd\,,\cd)$ satisfies {\rm(H0)$_p'$}, $\psi_i(\cd)\in L^p_{\cF_T}(\Om;\dbR^n)$, and $(Y_i(\cd),Z_i(\cd\,,\cd))$ is adapted solution to the corresponding BSVIE, $i=1,2$, then
\bel{9.14}\ba{ll}
\ds\dbE|Y_1(t)-Y_2(t)|^p+\dbE\(\int_t^T|Z_1(t,s)-Z_2(t,s)|^2ds\)^{
p\over2}\\
\ns\ds\q\les K_p^0\dbE\Big|\psi_1(t)-\psi_2(t)+\int_t^T[g_1(t,s,Z_1(t,s))-g_2(t,s,Z_1(t,s))]
ds\Big|^p,\q\ae\;t\in[0,T].\ea\ee
\end{proposition}

\begin{proof} \rm First, we write the equation as follows:
$$Y(t)=\psi(t)+\int_t^Tg_1(t,s,0)ds+\int_t^T[g_1(t,s,Z(t,s))-g_1(t,s,0)]ds-\int_t^TZ(t,s)dW(s).$$
For almost all $t\in[0,T]$, consider BSDE:
\bel{BSDE-I0}\eta(t,r)=\psi(t)+\int_t^Tg_1(t,s,0)ds+\int_r^T[g_1(t,s,\z(t,s))-g_1(t,s,0)]ds-\int_r^T\z(t,s)dW(s),\q r\in[t,T].\ee
Then regarding the sum of the first two terms on the right-hand side as the terminal state, and note that the generator $(s,\z)\mapsto g_1(t,s,\z)-g_1(t,s,0)\equiv\wt g_1(t,s,\z)$ has the property that $\wt g_1(t,s,0)=0$, satisfying (H0)$_p$, one has
\bel{}\ba{ll}
\ds\[\sup_{r\in[t,T]}\dbE|\eta(t,r)|^p\]^{1\over p}+\[\dbE\(\int_t^T|\z(t,s)|^2ds\)^{p\over2}\]^{1\over p}\les \h K_p\(\dbE\Big|\psi(t)+\int_t^Tg_1(t,s,0)ds\Big|^p\)^{1\over p}.\ea\ee
By setting
$$Y(t)=\eta(t,t),\q Z(t,s)=\z(t,s),$$
we have that $(Y(\cd),Z(\cd\,,\cd))$ is an adapted solution to BSVIE \rf{BSVIE-I0}, and
$$\begin{aligned}
&\[\dbE|Y(t)|^p+\dbE\(\int_t^T|Z(t,s)|^2ds\)^{p\over2}\]^{1\over p}\les\[\dbE|Y(t)|^p\]^{1\over p}+\[\dbE\(\int_t^T|Z(t,s)|^2ds\)^{p\over2}\]^{1\over p}\\
&\q\les \h K_p\(\dbE\Big|\psi(t)+\int_t^Tg_1(t,s,0)ds\Big|^p\)^{1\over p}.\end{aligned}$$
Then, \rf{9.13} holds. The uniqueness follows from  \rf{K<1-2}.
Next, let $g_i(\cd\,,\cd)$ satisfy (H0)$_p'$, $\psi_i(\cd)\in L^p_{\cF_T}(\Om;\dbR^n)$, and let $(Y_i(\cd),Z_i(\cd\,,\cd))$ be the corresponding BSVIE. Then
$$\ba{ll}
\ns\ds Y_1(t)-Y_2(t)=\psi_1(t)-\psi_2(t)+\int_t^T\(g_1(t,s,Z_1(t,s))-g_2(t,s,Z_1(t,s))\)ds\\
\ns\ds\qq\qq\qq\q+\int_t^TG(t,s)[Z_1(t,s)-Z_2(t,s)]ds-\int_t^T[Z_1(t,s)-Z_2(t,s)]dW(s),\qq t\in[0,T],\ea$$
where,
$$G(t,s)={[g_2(t,s,Z_1(t,s))-g_2(t,s,Z_2(t,s))][Z_1(t,s)-Z_2(t,s)]^\top\over|Z_1(t,s)-Z_2(t,s)|^2}{\bf1}_{\{
Z_1(t,s)\ne Z_2(t,s)\}}.$$
It is clear that
$$|G(t,s)|\les L_z^1(t,s),\qq\forall(t,s)\in\D[0,T].$$
Then, from the above, one has
\bel{}\ba{ll}
\ds\[\dbE|Y_1(t)-Y_2(t)|^p+\dbE\(\int_t^T|Z_1(t,s)-Z_2(t,s)|^2ds\)^{p\over2}\]^{1\over p}\\
\ns\ds\q\les\h K_p\(\dbE\Big|\psi_1(t)-\psi_2(t)+\int_t^T[g_1(t,s,Z_1(t,s))-g_2(t,s,Z_1(t,s))]ds\Big|^p\)^{1\over p}.\ea\ee
Hence, \rf{9.14} follows.
\end{proof}

\ms

We remark that \autoref{simple BSVIE} still holds true if  \rf{K<1-2} is replaced by
\bel{}\sup_{t\in[0,T]}\(\int_t^TL^1_z(t,s)^{2+\e}ds\)<\infty.\ee
with $\e>0$.
The following is concerned with the property of $g_1(\cd)$ in the decomposition \rf{g=g_0+g_1}.

\ms

Let $h:[0,T]\times\dbR\times\Om\to\dbR$ be $\cB([0,T]\times\dbR\times\cF_T)$-measurable and $h(s,y)$ be continuous with respect to the variable $y$.
Moreover, suppose that $h(\cd\,,0)\in L^1_{\cF_T}(0,T;\dbR)$ and there exists a function $L(\cd)\in L^\infty(0,T;\dbR_+)$ such that
$$|h(s,y)|\les |h(s,0)|+L(s)|y|.$$
Let
$$\h h(s,y)=\dbE_s[h(s,y)],\q \forall (s,y)\in[0,T]\times\dbR.$$
Then the mapping $y\mapsto \h h(s,y,\om)$ is continuous,
and for any fixed $y\in\dbR$, the process $\h h(\cd\,,y)$ is $\dbF$-adapted.

\ms

\bl{7.4} \sl
For any $Y(\cd)\in L^1_\dbF(0,T;\dbR)$, the process $\h h(s, Y(s));s\in[0,T]$ is $\dbF$-adapted.
\footnote{This result is obtained by a useful discussion with Chenchen Mou (City University of Hong Kong).}
\el

\begin{proof}
By the definition of $\dbF$-adapted processes, we need only  to prove that for any $s\in[0,T]$, $\h h(s,Y(s))\in\cF_s$.
Then it suffices to show that for any $O=(a,b)\subseteq\dbR$,
$$
\{\om\in\Om\,|\,\h h(s, Y(s,\om),\om)\in \bar O=[a,b] \}\in\cF_s,
$$
because $(a,\infty)=\bigcup_{k,n=1}^\infty[a+1/n,a+k]$.
For the ease of presentation, we write $\{\om\in\Om\,|\,\h h(s, Y(s,\om),\om)\in O \}$ as $\{\h h(s, Y(s))\in O \}$.

\ms

It is clearly seen that
\bel{}
\{\h h(s, Y(s))\in\bar O \}=\bigcup_{y\in\dbR}\big(\{\h h(s, y)\in\bar O \}\cap\{Y(s)=y\}\big).
\ee
Let $A$ be a countable dense set of $R$.
We claim that
\bel{claim}
\bigcup_{y\in\dbR}\big(\{\h h(s, y)\in \bar O \}\cap\{Y(s)=y\}\big)
=\bigcap_{m=1}^\infty \bigcap_{n=1}^\infty \bigcup_{y\in A}\(\big\{\h h(s, B_{1\over n}(y))\in O_m \big\}\cap\big\{Y(s)\in B_{1\over n}(y)\big\}\),
\ee
where
\begin{align*}
&\big\{\h h(s, B_{1\over n}(y))\in O_m \big\}=\bigcup_{x\in B_{1\over n}(y)}\big\{\h h(s, x)\in O_m \big\}\q\hbox{and}\q O_m=(a-1/m,b+1/m).
\end{align*}
Since the mapping $y\mapsto h(s,y)$ is continuous and $O_m$ is an open set, we have
$$
\big\{\h h(s, B_{1\over n}(y))\in O_m \big\}=\bigcup_{x\in B_{1\over n}(y)\cap A}\big\{\h h(s, x)\in O_m \big\}.
$$
If the claim \rf{claim} holds, the above implies that
$$
\bigcup_{y\in\dbR}\big(\{\h h(s, y)\in\bar O \}\cap\{Y(s)=y\}\big)
=\bigcap_{m=1}^\infty\bigcap_{n=1}^\infty \bigcup_{y\in A}\(\bigcup_{x\in B_{1\over n}(y)\cap A}\big\{\h h(s, x)\in O_m \big\}\bigcap\big\{Y(s)\in B_{1\over n}(y)\big\}\).
$$
Then from the facts  $\{\h h(s, x)\in O_m \}\in\cF_s$, $\{Y(s)\in B_{1\over n}(y)\}\in\cF_s$, and that $A$ is a countable set, we get
$$
\{\h h(s, Y(s))\in \bar O \}\in \cF_s.
$$

\ms

Now let us prove that the claim \rf{claim} really holds.

\ms
\emph{Step 1.}
$
\bigcup_{y\in\dbR}\big(\{\h h(s, y)\in\bar O \}\cap\{Y(s)=y\}\big)
\subseteq\bigcap_{m=1}^\infty \bigcap_{n=1}^\infty \bigcup_{y\in A}\big(\big\{\h h(s, B_{1\over n}(y))\in O_m \big\}\cap\big\{Y(s)\in B_{1\over n}(y)\big\}\big).
$

\ms
 For any $\om\in\bigcup_{y\in\dbR}\big(\{\h h(s, y)\in\bar O \}\cap\{Y(s)=y\}\big)$,
there exists  a $y\in\dbR$ such that $\h h(s, y,\om)\in\bar O$ and $Y(s,\om)=y$.
For any $m>0$ and $n>0$, $\bar O\subseteq O_m$ and there exists a $y_{mn}\in A$ such that $y\in B_{1\over n}(y_{mn})$.
Then
$$
\om\in\big\{\h h(s, B_{1\over n}(y_{mn}))\in O_m \big\}\cap\big\{Y(s)\in B_{1\over n}(y_{mn})\big\}.
$$
Thus, ``$\subseteq$" holds.

\ms

\emph{Step 2.}
$
\bigcup_{y\in\dbR}\big(\{\h h(s, y)\in \bar O \}\cap\{Y(s)=y\}\big)
\supseteq\bigcap_{m=1}^\infty \bigcap_{n=1}^\infty \bigcup_{y\in A}\big(\big\{\h h(s, B_{1\over n}(y))\in O_m \big\}\cap\big\{Y(s)\in B_{1\over n}(y)\big\}\big).
$

\ms

Let $\om\in\bigcap_{m=1}^\infty \bigcap_{n=1}^\infty\bigcup_{y\in A}\big(\big\{\h h(s, B_{1\over n}(y))\in O_m \big\}\cap\big\{Y(s)\in B_{1\over n}(y)\big\}\big)$.
Then for any $m$ and $n$, there exists a $y_{m,n}\in A$ such that
$$
Y(s,\om)\in B_{1\over n}(y_{mn})\q\hbox{and}\q \exists z_{mn}\in B_{1\over n}(y_{mn}) \hbox{~~such~that~~} \h h(s, z_{mn},\om)\in O_m.
$$
Note that
$$
|y_{mn}-Y(s,\om)|\les {1\over n}\q \hbox{and}\q |y_{mn}-z_{mn}|\les {1\over n}.
$$
By the continuity of $\h h(s,y)$ with respect to $y$, we get
$$
\h h(s, Y(s,\om),\om)\in \bigcap_{m=1}^\infty O_m= \bar O.
$$

\end{proof}

\end{document}